\documentclass[10pt]{article}
\usepackage[paperwidth=8.1in,paperheight=11in,top=1.1in, bottom=1.11in, left=1.00in, right=1.00in]{geometry}
\usepackage{amsthm,amsmath,amssymb,mathtools}  
\usepackage{chngcntr}
\usepackage{amsmath}  
\usepackage{graphicx}                                                         
\usepackage{caption}
\usepackage{subfigure}
\usepackage{enumitem}                                                       
\usepackage{hyperref}
\usepackage[mathscr]{eucal}                                            
\usepackage[usenames,dvipsnames,svgnames,table]{xcolor}
\usepackage[normalem]{ulem}                                         
\usepackage{pstricks}
\usepackage{rotating}
\usepackage{lscape}
\usepackage{dsfont}
\usepackage{dsfont}
\mathtoolsset{showonlyrefs=true}                                    
\theoremstyle{plain}
\usepackage{relsize}
\newtheorem{theorem}{Theorem}
\numberwithin{theorem}{section}

\newtheorem{lemma}[theorem]{Lemma}          
 \newtheorem{proposition}[theorem]{Proposition}
\newtheorem{corollary}[theorem]{Corollary}
\theoremstyle{definition}

\newtheorem{remark}[theorem]{Remark}
\newtheorem{assumption}[theorem]{Assumption}

\DeclarePairedDelimiter\ceil{\lceil}{\rceil}

\newcommand\Eb{\mathbb{E}}

\newcommand\Rb{\mathbb{R}}

\newcommand\R{\Rb}
\newcommand\E{\Eb}

\newcommand\Lc{\mathscr{L}}

\newcommand\dd{\mathrm{d}}
\newcommand\ee{\mathrm{e}}

\newcommand\ess{\textrm{ess}}
\newcommand\ind{\mathbf{1}}

\newcommand\eps{\varepsilon}

\def \s {{\sigma}}
\def \g {{\gamma}}
\def \a {{\alpha}}
\def \b {{\beta}}

\definecolor{blue}{rgb}{0,0,1}

\definecolor{violet}{rgb}{1,0,1}

\newcommand \aX{\bar{\alpha}}
\newcommand \bX{\bar{\beta}}

\newcommand{\eqlnostar}[2]{\begin{align}\label{#1}#2\end{align}}
\newcommand{\eqstar}[1]{\begin{align*}#1\end{align*}}

\begin{document}

\title{A Fourier-based Picard-iteration approach for a class of McKean-Vlasov SDEs with L\'evy jumps}

\author{
Ankush Agarwal
\thanks{Adam Smith Business School, University of Glasgow, University Avenue, G128QQ Glasgow, United Kingdom. Email: {\tt  ankush.agarwal@glasgow.ac.uk}.}
\and
Stefano Pagliarani
\thanks{Dipartimento di Scienze Economiche e Statistiche, Universit\`{a} di Udine, Via Tomadini 30/a, 33100 Udine, Italy. Email: {\tt stefano.pagliarani@uniud.it}.}
}
\date{This version: \today}
\maketitle

\begin{abstract}
We consider a class of L\'{e}vy-driven stochastic differential equations (SDEs) with McKean-Vlasov (MK-V) interaction in the drift coefficient. It is assumed that the coefficient is bounded, affine in the state variable, and only measurable in the law of the solution. We study the equivalent functional fixed-point equation for the unknown time-dependent coefficients of the associated Markovian SDE. By proving a contraction property for the functional map in a suitable normed space, we infer existence and uniqueness results for the MK-V SDE, and derive a discretized Picard iteration scheme that approximates the law of the solution through its characteristic function. Numerical illustrations show the effectiveness of our method, which appears to be appropriate to handle the multi-dimensional setting.
\end{abstract}

\section{Introduction}

We study a particular class of McKean-Vlasov (MK-V) stochastic differential equations (SDEs) where the coefficients are functions of both the state variable and the law of the solution. Introduced by McKean in the $60$'s, these equations have received increasing attention in the last few decades due to their wide range of applications in several fields, which include physics, neurosciences, economics and finance among others. In particular, the link with mean-field interacting particle systems, and the advent of mean-field games, have boosted the research on MK-V SDEs. 

From the numerical perspective, the study of the solutions to MK-V SDEs has been mainly conducted by exploiting the so-called ``propagation of chaos" results. It is usually shown that, in the limit $N \to \infty$, the empirical law of a Markovian system with $N$ interacting particles, converges to the law of the solution to the MK-V SDE under study, which can be then approximated via time-discretization and simulation. Following the work by Sznitman in \cite{Sznitman91}, where the first propagation of chaos result was proved, many authors have contributed to this stream of literature, see, for example, \cite{anto:koha:02}, \cite{tala:vail:03}, \cite{haji2018multilevel}, and \cite{sun2017ito}. Although a very powerful approximating tool, the simulation of large-particle systems can be computationally very expensive. For this reason, several authors investigated alternative approaches to the resolution of MK-V SDEs. Szpruch et.\ al.\ \cite{szpruch2017iterative} provided an alternative iterative particle representation that can be combined with Multilevel Monte Carlo techniques in order to simulate the solutions. In \cite{sun2017ito}, Sun et.\ al.\ developed It\^{o}-Taylor schemes of Euler- and Milstein-type for numerically estimating the solution of MK-V SDEs with Lipschitz regular coefficients and square integrable initial law. Gobet and Pagliarani \cite{gobet2016analytical} recently developed analytical approximations of the transition density of the solutions by extending a perturbation technique that was previously developed for standard SDEs. In \cite{chau:garc:15}, Chaudru De Raynal and Garcia Trillos developed a cubature method to obtain estimates for the solution of forward-backward SDEs of MK-V type. Furthermore, in \cite{belomestny2017projected}, Belomestny and Schoenmakers developed a novel projection-based particle method and tested it for a class of MK-V SDEs that includes some of the cases considered in this paper. 

In this work, we propose a Picard-iteration scheme for a class of MK-V SDEs driven by a L\'{e}vy process. Precisely, we assume a linear mean-field interaction through the law (via expectation of a measurable function), with a drift coefficient that is affine in the state variable. We reformulate the problem as an equivalent fixed-point equation for the unknown time-dependent coefficients of the related Markovian SDE. Exploiting a priori estimates on the characteristic function of the solution, we show that the related functional map is a contraction. This implies existence and uniqueness of the solution to the MK-V SDE. We then discretize the functional map in order to obtain a fully implementable Picard-iteration scheme to accurately approximate the functions that determine the law of the solution. We also provide the rate of convergence of our scheme with respect to both the time-discretization step and number of Picard iterations, and show that it is independent of the dimension.

In our setting, we consider a general underlying L\'{e}vy process and a general initial law. To deal with this level of generality for the L\'{e}vy measure, we show the contraction result by proving suitable estimates in the Fourier space. For this reason, we need a weak assumption on integrability of the initial moment and the L\'{e}vy measure, namely finiteness of the first moment. Also, our initial assumptions require the drift coefficients, as well as their Fourier transforms, to be in $L^1.$ However, we are able to overcome this limitation by using the so-called \emph{damping method} (see, for example, \cite{dubner1968numerical}). By a suitable modification of the functions appearing in the drift coefficients using penalization functions, and exploiting the properties of Fourier transform derivatives, we are able to prove the contraction property and all the consequent results in the case of non-integrable (though bounded) coefficients. However straightforward and effective, this strategy to circumvent the $L^1$ assumption on the coefficients does not seem completely satisfactory. In fact, on the one hand, this generalization requires additional hypotheses, namely higher order integrability of the underlying L\'{e}vy process and the initial law. On the other hand, while it is possible to drop the integrability assumption on the coefficients, we cannot avoid requiring the Fourier transform of the \emph{damped} (penalized) coefficients being in $L^1$, save the the case $d=1$ for which the latter assumption can be also dropped. It is important to point out that all these restrictions do not have a clear probabilistic interpretation, but they rather seem to be related to our choice of carrying out the analysis in the Fourier space. Although this approach offers a great deal of advantage in that it allows us to deal with a generic L\'evy measure, it may not be the optimal choice in order to prove the contraction property under minimal assumptions, at least in some specific cases. In order to substantiate this claim, in Section \ref{example: multiple dimensions 2} we demonstrate numerically that our discretized Picard iteration method converges with the expected rate even when choosing an initial law with infinite first moment. We aim to come back on this point in a future work, where we plan to extend the approach to more general coefficients (not necessarily linear in the state variable) by performing the analysis in the original space as opposed to the Fourier one. 

Although the class of MK-V SDEs approximated in this paper is relatively small, we point out that the idea of translating the mean-field SDE into a fixed-point equation on the coefficients space turns out to be effective, and leads to a Picard iteration method that is numerically efficient. Moreover, it allows for generic L\'{e}vy jumps in the dynamics. To the best of our knowledge, most numerical methods available in the literature for MK-V equations with L\'evy jumps, make use of propagation of chaos results and simulation of the related large-particle system (see, for example, \cite{andreis2017mckean} \cite{jourdain08Levy} and \cite{graham1992mckean}). Moreover, even though the existence-uniqueness results for MK-V SDEs available in the literature have reached a great level of generality that goes far beyond the setting considered here, most existing results make use of Wasserstein metrics, and thus require finiteness of first or second moments of the initial distribution (see, for example, \cite{jourdain08Levy} and \cite{hao2016mean} in the case of L\'evy jumps). As discussed above, our approach shows that integrability of the initial law is not crucial for the class considered here, neither for what concerns the existence nor numerical approximation of the solutions. 

The Picard-iteration approach developed here paves the way to handle more general instances of MK-V SDEs. In particular, a more general dependence (nonlinear) of the coefficients on both law and state-variable could be considered, by making use of parametrix estimates on the density kernels in both purely- and jump-diffusion settings. Another interesting extension consists in dropping the ellipticity hypothesis on the diffusion coefficient, so as to consider degenerate MK-V diffusions. This would involve the study of time-dependent H\"ormander-type conditions. Along with the above extensions, in future works, we also aim to consider the case of common noise, such as common Brownian motion and/or Levy jumps, for which propagation of chaos results have been recently obtained in \cite{andreis2017mckean}.  

This paper is organized as follows: In Section \ref{section:mckean SDE}, we introduce the MK-V SDE under study and provide the results on characteristic function of the solution. In Section \ref{section:contraction property}, we prove the contraction property of the functional mapping which provides the existence and uniqueness result of the solution. We prove the contraction property for $L^1$ coefficients in Section \ref{section:contraction L1}. We introduce the damping method and prove the contraction property for bounded functions in Section \ref{sec:damping}. The discretized Picard iteration scheme to obtain solution estimates and its rate of convergence is presented in Section \ref{section: discrete Picard}. Numerical experiments to validate our theoretical results are illustrated in Section \ref{section:numerics}. The proofs of intermediate lemmas and a priori estimates are relegated to Appendix \ref{app:proof_ito} and \ref{app:estimates_apriori} respectively. 

\section{Linear MK-V SDEs with jumps}
\label{section:mckean SDE}
Let us consider the following MK-V SDE on $\R^d$
\begin{equation}\label{eq:mckean_SDE_jumps}
\left\{
\begin{aligned}
\dd X_t  &=  \Eb[ a(X_t) x + b(X_t)]|_{x=X_t} \dd t + \dd L_t , \quad  t>0, \\
X_0 &= Y,
\end{aligned}
\right. 
\end{equation}
where $Y$ is a $\R^d$-valued random variable,  $a:\R^d\to \mathcal{M}^{d\times d}$, $b:\R^d\to \R^{d}$, and $L$ is a $d$-dimensional L\'{e}vy process with characteristic triplet $\big(0, \theta, \nu(\dd y)\big)$, meaning that 
\begin{equation}\label{eq:levy_proc}
L_t = \sigma W_t +\int_0^t \int_{|y| \geq 1} y  N(\dd s,\dd y) + \int_0^t \int_{|y|< 1} y  \big( N(\dd s,\dd y) - \nu(\dd y)\dd s \big) .
\end{equation}
In the above, ${N}(\dd s,\dd y)$ is a $d$-dimensional Poisson measure with compensator $\nu(\dd y)\dd s$, $W$ is a $q$-dimensional Brownian motion, and $\sigma$ is a $d\times q$ matrix such that $\s \s^\top = \theta$ is positive definite. 

We observe that solving \eqref{eq:mckean_SDE_jumps} up to time $T>0$ is equivalent to finding measurable functions $\alpha:[0,T] \to \mathcal{M}^{d\times d} $ and $\beta:[0,T] \to \mathbb{R}^{d} $ that solve the following MK-V fixed-point equation
\begin{equation}
\label{eq:ste15}
(\a_t,\b_t) =    \Big( \Eb \big[ a\big(X^{(\alpha,\beta)}_t\big) \big], \Eb \big[ b\big(X^{(\alpha,\beta)}_t\big) \big]  \Big)  =: \Psi_{t}(\a,\b), \quad t\in[0,T],
\end{equation}
where $X^{(\alpha,\beta)}$ denotes the solution to
\begin{equation}\label{eq:mckean_SDE_particular}
\left\{
\begin{aligned}
\dd X^{(\alpha,\beta)}_t   &= \big( \alpha_t X^{(\alpha,\beta)}_t  +\beta_t  \big) \dd t + \dd L_t , \quad  t\in]0,T], \\
X^{(\alpha,\beta)}_0 &=Y.
\end{aligned}
\right. 
\end{equation}
Clearly, $(\a,\b)$ solves \eqref{eq:ste15} if and only if $X^{(\a,\b)}$ solves MK-V SDE \eqref{eq:mckean_SDE_jumps}. Since the distribution of $X^{(\a,\b)}_t$ can be analytically characterized if $(\a,\b)\in L^{\infty}([0,T])$, we will look for solutions in this class. Thus, we will look at $\Psi=(\Psi_1,\Psi_2)$ defined in \eqref{eq:ste15} as an operator from\footnote{$L^{\infty}([0,T]:\mathcal{M}^{d\times d}\times \R^d)$ denotes the space of $L^{\infty}([0,T])$ functions with values on $\mathcal{M}^{d\times d}\times \R^d$.} $L^{\infty}([0,T]:\mathcal{M}^{d\times d}\times \R^d)$ onto itself which is equipped with the family of norms
\eqstar{
\|\gamma\|_{T,\lambda} := \max\{\|\alpha\|_{T,\lambda},\|\beta\|_{T,\lambda}\}, \, \lambda>0, \quad \gamma=(\alpha,\beta)\in L^{\infty}([0,T]:\mathcal{M}^{d\times d}\times \R^d),
}
where
\eqstar{
\|\alpha\|_{T,\lambda} &:= \ess\sup_{t\in [0,T]}  \ee^{-\lambda t}|\alpha_t| , \quad \alpha\in L^{\infty}([0,T]:\mathcal{M}^{d\times d}),\\
\|\beta\|_{T,\lambda} &:=\ess\sup_{t\in [0,T]}  \ee^{-\lambda t}|\beta_t|  ,  \quad \beta\in L^{\infty}([0,T]:\R^d),
}
Above, $|\cdot|$ represents either the spectral or Euclidean norm, depending on whether it is applied to an element of $\mathcal{M}^{d\times d}$ or $\R^d$, respectively. Hereafter, we denote the marginal law 
of the solution to \eqref{eq:mckean_SDE_particular}  by $\mu_{X^{(\a,\b)}_t}.$
\subsection{Preliminaries on linear SDEs with jumps}
\label{section:prelims}

For any probability measure $\mu$ on $\mathcal{B}(\R^d)$, we define its Fourier transform as
\begin{equation}
\hat{\mu}(\eta) :=\int_{\R^d}\ee^{i \langle \eta , x\rangle} \mu( \dd x) ,\quad \eta\in\R^d.
\end{equation}
Analogously, for any function $f\in L^1(\R^d)$ we define its the Fourier transform as
\begin{equation}
\hat{f}(\eta) :=\int_{\R^d} \ee^{i \langle\eta, x\rangle} f(x) \dd x, \quad \eta\in\R^d,
\end{equation}
and, the inverse Fourier transform as
\eqstar{
f(x) =\frac{1}{(2 \pi)^d} \int_{\R^d}\ee^{-i \langle\eta, x\rangle} \hat{f}(\eta) \dd \eta, \quad x\in\R^d.
}
The inverse Fourier transform makes sense if $\hat{f}\in L^1(\R^d).$ In general, we will resort to Plancherel's theorem which states that if $f,g\in L^1(\R^d)\cap L^2(\R^d)$, then $\hat{f},\hat{g}\in L^2(\R^d)$ and 
\begin{equation}
\int_{\R^d} f(x) g(x) \dd x = \frac{1}{(2\pi)^d} \int_{\R^d} \hat{f}(\eta)  \hat{g}(-\eta) \dd \eta. 
\end{equation}

In Lemma \ref{lem:ito} below, we have the first preliminary result regarding the Fourier transform of $\mu_{X^{(\a,\b)}_t}.$ It is a standard result  but due to the lack of a precise statement in the literature, we provide the proof, which is reported in  Appendix \ref{app:proof_ito}.
\begin{lemma}
\label{lem:ito}
For any $T>0$ and $(\alpha,\beta)\in L^{\infty}([0,T]:\mathcal{M}^{d\times d}\times \R^d)$, the Fourier transform of the law $\mu_{X^{(\a,\b)}_t}$ is given as
\begin{equation}\label{eq:ste13}
\hat{\mu}_{X^{(\a,\b)}_t}( \eta) = \exp\Big(\! -\frac{1}{2}\langle \eta, C^{\alpha}_t\eta\rangle + i\,\langle  \eta, m^{\alpha,\beta}_t\rangle +   n^{\alpha}_t(\eta)  \Big) \hat{\mu}_Y \big(  (\Phi^{\alpha}_{0,t})^\top \eta\big),
\end{equation}
where
\eqlnostar{eq:ste10}{
m^{\alpha,\beta}_t &=   \int_0^t  \Phi^{\alpha}_{s,t} \b_s \dd s, \qquad C^{\alpha}_t=   \int_0^t \Phi^{\alpha}_{s,t} \,\s \s^\top (\Phi^{\alpha}_{s,t})^\top  \,\dd s,\\
\label{eq:ste10bis}
n^{\alpha}_t(\eta) &= \int_{0}^t \int_{\R^d} \Big( \ee^{i \langle  \eta , \Phi^{\alpha}_{s,t} y\rangle} -1 -  {\bf 1}_{\{|y|<1\}} i\big\langle   \eta, \Phi^{\alpha}_{s,t} y\big\rangle  \Big) \nu(\dd y) \dd s.
}
In the above, $\Phi^{\alpha}_{s,\cdot}:[s,T]\to \mathcal{M}^{d\times d}$ is the unique solution of 
\begin{equation}\label{eq:ode}
\left\{
\begin{aligned}
\frac{\dd}{\dd t}\Phi^{\alpha}_{s,t}  &=\a_t \Phi^{\alpha}_{s,t} , \quad s<t\leq T,\\
\Phi^{\alpha}_{s,s} &= I_d, 
\end{aligned}
\right. 
\end{equation}
which is an absolutely continuous function such that 
$
\Phi^{\alpha}_{s,t} = I_d + \int_{s}^{t} \alpha_u \Phi^{\alpha}_{s,u} \dd u
$, for $0\leq t\leq T$.
\end{lemma}

We also have the following estimates on the quantities arising in $\hat{\mu}_{X^{(\a,\b)}_t}$, which will be used repeatedly throughout the paper, and whose proofs are also postponed until Appendix \ref{app:proof_ito}.
\begin{lemma}\label{lemm:estimates_bis}
For any $T>0, 0<\delta< 1$ and $(\alpha,\beta)\in L^{\infty}([0,T]:\mathcal{M}^{d\times d}\times \R^d)$, we have 
\begin{align}
|\Phi^{\a}_{s,t}|& \leq \ee^{T \|\a \|_{T,0} },\label{eq:est_1}\\
| C^{\alpha}_t | &\leq t |\theta| \ee^{2 T \|\alpha\|_{T,0} },\label{eq:est_1_b}\\
| m^{\alpha,\beta}_t | &\leq t  \ee^{ T \|\alpha\|_{T,0} } \|\beta\|_{T,0},\label{eq:est_1_c}\\ \label{eq:estim_n}
|n^{\alpha}_t(\eta)| &\leq 2 t \ee^{ 2T \|\alpha\|_{T,0} }\bigg(  |\eta|^2  \int_{|y|< \delta} |y|^2 \nu(\dd y)  + (|\eta|+1) \int_{|y|\geq \delta}  \nu(\dd y) \bigg) ,  \\ \label{eq:estimate_quad_form}
|\langle \eta, C^{\alpha}_t\eta\rangle|& \geq \lambda_{\text{min}}(\theta) \ee^{-2T \| \alpha \|_{T,0}} t |\eta|^2,
\end{align}
for any $0\leq s \leq t\leq T$ and $\eta\in\R^d$.
In the above, $\lambda_{\text{min}}(\theta)$ denotes the minimum eigenvalue of $\theta$.
\end{lemma}
\begin{lemma}\label{lemm:estimates_ter}
For any $T>0$ and $(\alpha,\beta)\in L^{\infty}([0,T]:\mathcal{M}^{d\times d}\times \R^d)$, we have that
\begin{align}\label{eq:est_1_ter}
|\Phi^{\a}_{s,t'} - \Phi^{\a}_{s,t}|& \leq (t'-t)\|\a \|_{T,0} \ee^{T \|\a \|_{T,0}} ,\\
| C^{\alpha}_{t'} - C^{\alpha}_{t} | &\leq (t'-t) |\theta| \ee^{4 T \|\a \|_{T,0}},   \label{eq:est_1_b_ter}\\
| m^{\alpha,\beta}_{t'} - m^{\alpha,\beta}_{t} | &\leq (t'-t) \|\b \|_{T,0} \ee^{2 T \|\a \|_{T,0}} ,\label{eq:est_1_c_ter}
\intertext{
(and in addition, if Assumption \ref{ass:levy} holds)
}
|n^{\alpha}_{t'}(\eta) - n^{\alpha}_{t}(\eta)| &\leq 2(t'-t) \ee^{3 T \| \a \|_{T,0}}  \big( |\eta|^2 + |\eta| + 1  \big)\int_{\R^d}|y|(1\wedge |y|) \nu(\dd y),\label{eq:estim_n_ter}
\end{align}
for any $0\leq s \leq t\leq t' \leq T$ and $\eta\in\R^d$.
\end{lemma}

\section{Contraction property in the Fourier space}
\label{section:contraction property}
\subsection{The case $a,b\in L^1(\R^d)$}
\label{section:contraction L1}
In this section, we perform the analysis of the fixed-point equation \eqref{eq:ste15} under the simplified assumption that the coefficients $a,b$ are integrable and prove that the function $\Psi_t$ defined in equation \eqref{eq:ste15} is a contraction map. 
\begin{assumption}\label{assump:coeff}
The coefficients $a,b \in L^{\infty}(\R^d)\cap L^{1}(\R^d)$. 
\end{assumption}
\begin{assumption}\label{ellipticity}
The matrix $\theta= \sigma\sigma^\top$ is positive definite, that is, its smallest eigenvalue $\lambda_{\text{min}}(\theta)$ is positive. 
\end{assumption}

\begin{proposition}\label{prop_verification}[Fourier representation]
Under Assumptions \ref{assump:coeff} and \ref{ellipticity}, for any $T>0$ and $(\a,\b)\in L^{\infty}([0,T]:\mathcal{M}^{d\times d}\times \R^d),$ we have that
\begin{align}\label{eq:fourier_rep1}
{\Psi_{1,t}}(\alpha,\beta)&= \frac{1}{(2\pi)^d} \int_{\R^d} \hat{a}(\eta) \exp\bigg(\!{-\frac{1}{2}\langle \eta, C^{\alpha}_t\eta\rangle - i\,\langle  \eta, m^{\alpha,\beta}_t\rangle +   n^{\alpha}_t(-\eta)}\!\bigg) 
\hat{\mu}_Y \!\left(-(\Phi^{\alpha}_{0,t})^\top \eta\right)\dd \eta,\\ \label{eq:fourier_rep2}
{\Psi_{2,t}}(\alpha,\beta)&= \frac{1}{(2\pi)^d} \int_{\R^d} \hat{b}(\eta) \exp\bigg(\!{-\frac{1}{2}\langle \eta, C^{\alpha}_t\eta\rangle - i\,\langle  \eta, m^{\alpha,\beta}_t\rangle +   n^{\alpha}_t(-\eta)}\!\bigg) 
\hat{\mu}_Y \!\left(-(\Phi^{\alpha}_{0,t})^\top \eta\right)\dd \eta,
\end{align}
for any $t\in[0,T],$
where $C^{\alpha},m^{\alpha,\beta},n^{\alpha},\Phi^{\alpha}$ are as defined in \eqref{eq:ste10}-\eqref{eq:ste10bis}-\eqref{eq:ode}. 
\end{proposition}
\begin{proof}
By \eqref{eq:ste13}, together with \eqref{eq:estimate_quad_form} and \eqref{eq:estim_n} with $\delta$ suitably small, we have that $\hat{\mu}_{X^{(\a,\b)}_t} \in L^p(\R^d)$ for any $p\in\mathbb{N}$. In particular $\hat{\mu}_{X^{(\a,\b)}_t} \in L^1(\R^d)$, and thus, ${\mu}_{X^{(\a,\b)}_t}$ has a density $f_{X^{(\a,\b)}_t}$ 
 that is given by 
\begin{equation}\label{eq:density}
f_{X^{(\a,\b)}_t} (x) = \frac{1}{(2\pi)^d} \int_{\R^d}  \ee^{-i \langle x,\eta \rangle}  \hat{\mu}_{X^{(\a,\b)}_t}(\eta)  \dd \eta, \quad x\in\R^d. 
\end{equation}
This yields
\begin{equation}\label{eq:density_rep}
{\Psi_{1,t}}(\alpha,\beta) = \int_{\R^d} a(x) {\mu}_{X^{(\a,\b)}_t} (\dd x) = \int_{\R^d} a(x) f_{X^{(\a,\b)}_t} (x) \dd x .
\end{equation}
Moreover, since $f_{X^{(\a,\b)}_t}(x)  = \hat{\hat{\mu}}_{X^{(\a,\b)}_t}(-x) $ and $\hat{\mu}_{X^{(\a,\b)}_t} \in L^1(\R^d)\cap L^2(\R^d)$, we also have that $f_{X^{(\a,\b)}_t} (x) $ $\in L^1(\R^d)\cap L^2(\R^d) $. Also, by Assumption \ref{assump:coeff}, $a\in L^{1}(\R^d)\cap L^{2}(\R^d)$. Therefore, by employing Plancherel's theorem in \eqref{eq:density_rep}, we get that
\begin{equation}
{\Psi_{1,t}}(\alpha,\beta) = \frac{1}{(2\pi)^d} \int_{\R^d} \hat{a}(\eta) \hat{\mu}_{X^{(\a,\b)}_t} (-\eta) \dd \eta.
\end{equation}
Finally, applying the result in Lemma \ref{lem:ito} concludes the proof of \eqref{eq:fourier_rep1}. The proof of \eqref{eq:fourier_rep2} is identical.
\end{proof}
\begin{remark}\label{rem:continuity_Psi}
Under Assumption \ref{assump:coeff} and \ref{ellipticity}, the function $t\to \Psi_t(\alpha,\beta)$ is continuous on $]0,T]$ for any $(\alpha,\beta)\in L^{\infty}([0,T]:\mathcal{M}^{d\times d}\times \R^d)$. This can be easily verified by using Proposition \ref{prop_verification} together with Lebesgue dominated convergence theorem.
\end{remark}
\begin{assumption}
\label{ass:levy}
The L\'evy measure $\nu$ is such that
\begin{equation}
\label{eq:ste52}
\bar{n}:= \int_{\R^d}|y|(1\wedge |y|) \nu(\dd y) <\infty,
\end{equation}
and initial datum $Y$ is such that
\begin{equation}
\Eb[|Y|] < \infty.
\end{equation}
\end{assumption}
\begin{assumption}
\label{assump:coeff_B}
The Fourier transforms 
$\hat{a},\hat{b}\in L^1 (\R^d)$.
\end{assumption}
\begin{theorem}\label{lem:contraction}[Contraction property]
Suppose that  Assumption \ref{assump:coeff}, \ref{ellipticity}, \ref{ass:levy} and \ref{assump:coeff_B} hold. Then, for any $T,c>0$, there exists $\lambda>0$, only dependent on $c$, $T$, $\theta$, $\| a \|_{L^{\infty}}$, $\| b \|_{L^{\infty}}$, $\frac{1}{(2\pi)^d}\| \hat{a} \|_{L^1}$, $\frac{1}{(2\pi)^d}\| \hat{b} \|_{L^1}$, $
\nu(\dd y)$ and $\Eb[|Y|]$, such that
\begin{align}\label{eq:contraction}
\|\Psi(\alpha,\beta) - \Psi(\alpha',\beta')\|_{T,\lambda}& \leq c    \|(\alpha -\alpha',\beta -\beta' )\|_{T,\lambda} ,
\end{align}
for any $(\alpha,\beta),(\alpha',\beta')\in L^{\infty}([0,T]:\mathcal{M}^{d\times d}\times \R^d)$ with $\| (\alpha,\beta)\|_{T,0} ,\| (\alpha',\beta')\|_{T,0} \leq   \| a \|_{L^{\infty}}+ \| b \|_{L^{\infty}}$.
\end{theorem}
In the one-dimensional case, Assumption \ref{assump:coeff_B} is not necessary. 
\begin{theorem}\label{lem:contraction_d1}
Suppose that  Assumption \ref{assump:coeff}, \ref{ellipticity} and \ref{ass:levy} hold, and $d=1$. Then, the same conclusions of Theorem \ref{lem:contraction} hold, with $\lambda$ that depends on $\| \hat{a} \|_{L^2}$, $\| \hat{b} \|_{L^2}$ instead of $\| \hat{a} \|_{L^1}$, $\| \hat{b} \|_{L^1}$.
\end{theorem}

\begin{corollary}\label{cor:fixed_point}
Under the assumptions of Theorem \ref{lem:contraction} (or Theorem \ref{lem:contraction_d1}), for any $T>0,$ we have that $\Psi$ is a contraction map from the set
\begin{equation}\label{eq:set_B}
B = \big\{ (\alpha,\beta)\in L^{\infty}([0,T]:\mathcal{M}^{d\times d}\times \R^d) :  \|(\alpha,\beta)\|_{T,0}\leq \| a \|_{L^{\infty}}+ \| b \|_{L^{\infty}}  \big\}
\end{equation}
onto itself, with respect to the norm $\|\cdot \|_{T,\lambda}$, with $\lambda$ as in Theorem \ref{lem:contraction} (or Theorem \ref{lem:contraction_d1}). In particular, there exists a unique solution $(\bar\alpha, \bar\beta)$ in $L^{\infty}([0,T]:\mathcal{M}^{d\times d}\times \R^d)$ to the MK-V fixed-point equation \eqref{eq:ste15}, and it is continuous on $]0,T]$. 
\end{corollary}
\begin{proof}
The fact that $a,b\in L^{\infty}(\R^d)$, and that ${X^{(\a,\b)}_t}$ has a density for any $t\in]0,T]$ and $(\alpha,\beta)\in L^{\infty}([0,T]:\mathcal{M}^{d\times d}\times \R^d)$, imply that $\Psi(\alpha,\beta) \in B$. Therefore, as a result of Theorem \ref{lem:contraction}, $\Psi$ is a contraction from $B$ onto itself, and \eqref{eq:ste15} has a unique solution in $L^{\infty}([0,T]:\mathcal{M}^{d\times d}\times \R^d)$, which belongs to $B$. 
The continuity of the solution is immediate from Remark \ref{rem:continuity_Psi}. 
\end{proof}
\begin{remark}
It is worth noticing that, under the assumptions of Theorem \ref{lem:contraction}, it is not excluded that the MK-V fixed-point equation \eqref{eq:ste15} has a solution that does not belong to $L^{\infty}([0,T])$. This is due to the fact that ${X^{(\a,\b)}_t}$ might have no density if $(\a,\b)\notin L^{\infty}([0,T]:\mathcal{M}^{d\times d}\times \R^d)$, and thus, it is not possible to establish a priori that  $\Psi(\alpha,\beta)$ belongs to $B$.
However, by reinforcing Assumption \ref{assump:coeff} which requires $a,b$ to be bounded instead of simply in $L^{\infty}(\R^d)$, the continuous solution $(\aX,\bX)$ in Corollary \ref{cor:fixed_point} is the only possible solution to \eqref{eq:ste15}. In fact, the boundedness of $a,b$ implies that $\Psi(\alpha,\beta) \in B$ for any choice of $(\alpha,\beta):[0,T] \to \mathcal{M}^{d\times d} \times \mathbb{R}^{d} $ such that \eqref{eq:mckean_SDE_particular} has a solution. Thus, a solution to \eqref{eq:ste15} necessarily belongs to $B$ and (by Remark \ref{rem:continuity_Psi}) is continuous on $]0,T]$. 
\end{remark}
The proof of Theorem \ref{lem:contraction} and Theorem \ref{lem:contraction_d1} is preceded by the following a priori estimates, whose proofs are postponed to Appendix \ref{app:estimates_apriori}.
\begin{lemma}
\label{lem:estimates}
Suppose that Assumption \ref{assump:coeff} and \ref{ass:levy} hold. 
For any $T>0$, and for any $(\alpha,\beta),(\alpha',\beta')\in L^{\infty}([0,T]:\mathcal{M}^{d\times d}\times \R^d)$ with $\|\alpha\|_{T,0},\|\alpha'\|_{T,0}\leq \| a \|_{L^{\infty}(\R^d)}$ and $\|\beta\|_{T,0} \leq \| b \|_{L^{\infty}(\R^d)}$, it holds that
\eqlnostar{eq:ste104}{
\big|\big\langle \eta, C^{\alpha}_t\eta\big\rangle - \big\langle \eta, C^{\alpha'}_t\eta\big\rangle  \big|   &\leq   \kappa\, t |\eta|^2 \int_0^t | \alpha_u - \alpha'_u | \dd u,\\ 
\label{eq:ste104bis}
\big| \big\langle  \eta, m^{\alpha,\beta}_t\big\rangle - \big\langle  \eta, m^{\alpha',\beta}_t\big\rangle  \big| & \leq   \kappa\, t |\eta| \int_0^t | \alpha_u - \alpha'_u | \dd u,\\ 
\label{eq:ste104ter}
\big| \big\langle  \eta, m^{\alpha,\beta}_t\big\rangle - \big\langle  \eta, m^{\alpha,\beta'}_t\big\rangle  \big| &\leq \kappa\, |\eta| \int_0^t | \beta_u - \beta'_u|   \dd u,\\
\label{eq:ste104_quin}
\big|\big(  (\Phi^{\alpha}_{s,t})^\top   -   (\Phi_{s,t}^{\alpha'})^\top  \big)\eta  \big|   &\leq   \kappa\, |\eta| \int_0^t | \alpha_u - \alpha'_u | \dd u,\\
\label{eq:ste104quat}
\big| n^{\alpha}_t(\eta) - n^{\alpha'}_t(\eta)  \big| &\leq   \kappa\, t |\eta|(1+|\eta|) \int_0^t | \alpha_u - \alpha'_u | \dd u,
}
for any $0\leq s\leq t\leq T, \eta\in\R^d$, where $\kappa$ is a positive constant that depends only on $T,\theta,\| a \|_{L^{\infty}}$, $\| b \|_{L^{\infty}}$, and $\nu(\dd y)$.
\end{lemma}

We are now ready to prove Theorems \ref{lem:contraction} and \ref{lem:contraction_d1}.
\begin{proof}[Proof of Theorem \ref{lem:contraction}]
Throughout this proof, we denote by $\kappa$ any positive constant that depends at most on  $T$, $\theta$, $\| a \|_{L^{\infty}}$, $\| b \|_{L^{\infty}}$, $\frac{1}{(2\pi)^d}\| \hat{a} \|_{L^1}$, $\frac{1}{(2\pi)^d}\| \hat{b} \|_{L^1}$, $
\nu(\dd y)$ and $\Eb[|Y|]$. 

\vspace{3pt}
\noindent\underline{Step 1}: For any $0\leq t\leq T$, we first prove that 
\begin{align}
& |\Psi_{1,t}(\alpha,\beta) -  \Psi_{1,t}(\alpha',\beta')|+|\Psi_{2,t}(\alpha,\beta) -  \Psi_{2,t}(\alpha',\beta')|\\
&\leq 
\kappa \bigg(\int_0^t |\a_s -\a'_s| \dd s + \int_0^t |\b_s -\b'_s| \dd s\bigg) \int_{\R^d} \big(|\hat{a}(\eta)| +|\hat{b}(\eta)| \big) \Bigl(t |\eta|^2 + |\eta| \bigl(t +1 \bigr)  \Bigr)\ee^{-\frac{ t|\eta|^2}{2 \kappa}}\dd \eta.
\label{eq:estim_Psi}
\end{align}
By Proposition \ref{prop_verification} and by triangular inequality, we obtain the following 
\begin{equation}
|\Psi_{1,t}(\alpha,\beta) -  \Psi_{1,t}(\alpha',\beta')| \leq \frac{1}{(2\pi)^d} \sum_{i=1}^5 I_i,
\end{equation}
where
\begin{align}
I_1 &=     \int_{\R^d} |\hat{a}(\eta)| \left| \ee^{-\frac{1}{2}\langle \eta, C^{\alpha}_t\eta\rangle} - \ee^{-\frac{1}{2}\langle \eta, C^{\alpha'}_t\eta\rangle } \right| \left| \ee^{- i\,\langle  \eta, m^{\alpha,\beta}_t\rangle +   n^{\alpha}_t(-\eta) } \right| \big|\hat{\mu}_Y \big(\!-\!(\Phi_{0,t}^{\alpha})^\top \eta\big)\big| \dd \eta     ,\\
I_2 &=   \int_{\R^d} |\hat{a}(\eta)| \left| \ee^{ -\frac{1}{2} \langle \eta, C^{\alpha'}_t \eta \rangle + n^{\alpha}_t (-\eta) }\right| \left|\ee^{- i\,\langle  \eta, m^{\alpha,\beta}_t\rangle } - \ee^{- i\,\langle  \eta, m^{\alpha',\beta}_t\rangle } \right|  \big|\hat{\mu}_Y \big(\!-\!(\Phi_{0,t}^{\alpha})^\top \eta\big)\big| \dd \eta     ,\\
I_3 &=   \int_{\R^d} |\hat{a}(\eta)| \left| \ee^{ -\frac{1}{2} \langle \eta, C^{\alpha'}_t \eta \rangle + n^{\alpha}_t (-\eta) }\right| \left|\ee^{- i\,\langle  \eta, m^{\alpha',\beta}_t\rangle } - \ee^{- i\,\langle  \eta, m^{\alpha',\beta'}_t\rangle } \right|  \big|\hat{\mu}_Y \big(\!-\!(\Phi_{0,t}^{\alpha})^\top \eta\big)\big| \dd \eta     ,\\
I_4 &=    \int_{\R^d} |\hat{a}(\eta)| \left| \ee^{ -\frac{1}{2} \langle \eta, C^{\alpha'}_t \eta \rangle  - i\,\langle  \eta, m^{\alpha',\beta'}_t\rangle }\right| \left|\ee^{n^{\alpha}_t (-\eta)  } - \ee^{ n^{\alpha'}_t (-\eta) } \right|  \big|\hat{\mu}_Y \big(\!-\!(\Phi_{0,t}^{\alpha})^\top \eta\big)\big| \dd \eta     ,\\
I_5 &=    \int_{\R^d} |\hat{a}(\eta)| \left| \ee^{ -\frac{1}{2} \langle \eta, C^{\alpha'}_t \eta \rangle  - i\,\langle  \eta, m^{\alpha',\beta'}_t\rangle + n^{\alpha'}_t (-\eta) }\right| \big|\hat{\mu}_Y  \big(\!-\!(\Phi_{0,t}^{\alpha})^\top \eta\big) - \hat{\mu}_Y  \big(\!-\!(\Phi_{0,t}^{\alpha'})^\top \eta\big) \big| \dd \eta      .
\end{align}
By \eqref{eq:ste104}, combined with \eqref{eq:estimate_quad_form}, we obtain that
\begin{equation}
 \left| \ee^{-\frac{1}{2}\langle \eta, C^{\alpha}_t\eta\rangle} - \ee^{-\frac{1}{2}\langle \eta, C^{\alpha'}_t\eta\rangle } \right| \leq 
 \kappa\, t |\eta|^2 \ee^{- \frac{t|\eta|^2}{2 \kappa}} \int_0^t | \alpha_s - \alpha'_s | \dd s ,
\end{equation}
while by \eqref{eq:estim_n}, for any $\eps>0,$ we have that
\begin{equation}\label{eq:estim_e_n}
\big| \ee^{   n^{\alpha}_t(-\eta) } \big| \leq \kappa\, \ee^{t(  \eps |\eta|^2    + \kappa  |\eta|)} \leq \kappa\, \ee^{2 \eps |\eta|^2 }.
\end{equation}
Taking $\eps$ suitably small in \eqref{eq:estim_e_n}, yields that
\begin{equation}
I_1 \leq \kappa \int_{\R^d}  |\hat{a}(\eta)|\, t |\eta|^2  \ee^{-\frac{ t|\eta|^2}{2 \kappa}}\dd \eta \int_0^t |\a_s -\a'_s| \dd s.
\end{equation}
Similarly, using \eqref{eq:ste104bis}-\eqref{eq:ste104ter}-\eqref{eq:ste104_quin}-\eqref{eq:ste104quat} and again \eqref{eq:estim_n}-\eqref{eq:estimate_quad_form}, and the fact that
\begin{equation}
\big|\hat{\mu}_Y  \big(\!-\!(\Phi_{0,t}^{\alpha})^\top \eta\big) - \hat{\mu}_Y  \big(\!-\!(\Phi_{0,t}^{\alpha'})^\top \eta\big) \big| \leq \big|\big(  (\Phi^{\alpha}_{0,t})^\top   -   (\Phi_{0,t}^{\alpha'})^\top  \big)\eta  \big|  \, \E [|Y|],
\end{equation}
we obtain that 
\begin{align}
I_2 &\leq \kappa \int_{\R^d}  |\hat{a}(\eta)|\, t |\eta|  \ee^{-\frac{ t|\eta|^2}{2 \kappa}}\dd \eta \int_0^t |\a_s -\a'_s| \dd s,\\
I_3 &\leq \kappa \int_{\R^d}  |\hat{a}(\eta)|\,  |\eta|  \ee^{-\frac{ t|\eta|^2}{2 \kappa}}\dd \eta \int_0^t |\b_s -\b'_s| \dd s,\\
I_4 &\leq \kappa \int_{\R^d}  |\hat{a}(\eta)|\, t |\eta|(1+|\eta|)  \ee^{-\frac{ t|\eta|^2}{2 \kappa}}\dd \eta \int_0^t |\a_s -\a'_s| \dd s,\\
I_5 &\leq \kappa \int_{\R^d}  |\hat{a}(\eta)|\,  |\eta|  \ee^{-\frac{ t|\eta|^2}{2 \kappa}}\dd \eta \int_0^t |\a_s -\a'_s| \dd s.
\end{align}
These, together with analogous estimates for $|\Psi_{2,t}(\alpha,\beta) -  \Psi_{2,t}(\alpha',\beta')|$, prove \eqref{eq:estim_Psi}. 

\vspace{3pt}
\noindent\underline{Step 2}: Combining \eqref{eq:estim_Psi} with the standard Gaussian estimate (see Lemma A.1. in \cite{gobet2016analytical})
\eqlnostar{eq:guassian_estimate}{
|\eta|^{\gamma} \ee^{- \frac{\rho |\eta|^2}{2} } \leq \Big(  \frac{2 \gamma}{e}  \Big)^{\frac{\gamma}{2}} \rho ^{-\frac{\gamma}{2}} \ee^{-\frac{\rho |\eta|^2}{4}},  \quad  \eta\in\R^d,\ \rho>0, \ \gamma\geq 0,
}
we obtain that
\begin{align}
& |\Psi_{1,t}(\alpha,\beta) -  \Psi_{1,t}(\alpha',\beta')| + |\Psi_{2,t}(\alpha,\beta) -  \Psi_{2,t}(\alpha',\beta')| \\ 
 \label{eq:estimate_psi_proof}
&\leq \frac{\kappa}{\sqrt{t}} \bigg(\int_0^t |\a_s -\a'_s| \dd s + \int_0^t |\b_s -\b'_s| \dd s\bigg)\biggl(\frac{1}{(2\pi)^d} \int_{\R^d} (|\hat{a}(\eta)|+|\hat{b}(\eta)|) \ee^{-\frac{ t|\eta|^2}{4 \kappa}}\dd \eta \biggr)
\intertext{(using $\hat{a}\in L^1(\R^d)$)}
& \leq \frac{\kappa}{\sqrt{t}} \bigg(\int_0^t |\a_s -\a'_s| \dd s + \int_0^t |\b_s -\b'_s| \dd s\bigg) , 
\end{align}
which for any $t\in[0,T],$ implies that
\begin{align}
\ee^{-\lambda t}  |\Psi_{1,t}(\alpha,\beta) -  \Psi_{1,t}(\alpha',\beta')| &\leq \frac{\kappa}{\sqrt{t}} \bigg(\int_0^t \ee^{-\lambda (t-s)} \ee^{-\lambda s} |\a_s -\a'_s| \dd s + \int_0^t \ee^{-\lambda (t-s)} \ee^{-\lambda s}|\b_s -\b'_s| \dd s\bigg)\\
&\leq  \big( \|  \a-\a' \|_{T,\lambda} + \|  \b-\b' \|_{T,\lambda} \big) \frac{\kappa}{\sqrt{\lambda}} \frac{(1-\ee^{-\lambda t})}{\sqrt{\lambda t}}\\
&\leq \big( \|  \a-\a' \|_{T,\lambda} + \|  \b-\b' \|_{T,\lambda} \big) \frac{\kappa}{\sqrt{\lambda}}.
\end{align}
The same estimate holds for $ |\Psi_{2,t}(\alpha,\beta) -  \Psi_{2,t}(\alpha',\beta')|$ and thus, taking $\lambda$ suitably large yields the result.
\end{proof}
\begin{proof}[Proof of Theorem \ref{lem:contraction_d1}]
Throughout this proof we denote by $\kappa$ any positive constant that depends at most on  $T$, $\theta$, $\| a \|_{L^{\infty}}$, $\| b \|_{L^{\infty}}$, $\frac{1}{(2\pi)^d}\| \hat{a} \|_{L^2}$, $\frac{1}{(2\pi)^d}\| \hat{b} \|_{L^2}$, $\nu(\dd y)$ and $\Eb[|Y|]$. 

Applying \eqref{eq:guassian_estimate} with $\gamma=2/3,$ for any $\rho>0$ and $\eta\in\R,$ we obtain that
\begin{align}
\int_{\R} (|\hat{a}(\eta)|+|\hat{b}(\eta)|) \ee^{-\frac{ \rho |\eta|^2}{2}}\dd \eta   &=  \int_{\R} \frac{|\hat{a}(\eta)|+|\hat{b}(\eta)|}{|\eta|^{2/3}} |\eta|^{2/3} \ee^{-\frac{ \rho|\eta|^2}{2}}\dd \eta  \\
&  \leq \Big(  \frac{4}{3 \ee}  \Big)^{\frac{1}{3}}  \rho^{-\frac{1}{3}} \int_{\R} \frac{|\hat{a}(\eta)|+|\hat{b}(\eta)|}{|\eta|^{2/3}}  \ee^{-\frac{ \rho|\eta|^2}{4}}\dd \eta   
\intertext{(by Plancherel's theorem)}
   &     \leq \Big(  \frac{4}{3 \ee}  \Big)^{\frac{1}{3}} \rho^{-\frac{1}{3}} \int_{\R} \big(|\hat{a}(\eta)|+|\hat{b}(\eta)|\big)^2 \dd \eta    \int_{\R} |\eta|^{-4/3} \dd \eta.
\end{align}
Therefore, \eqref{eq:estimate_psi_proof} yields that
\begin{equation}
 |\Psi_{1,t}(\alpha,\beta) -  \Psi_{1,t}(\alpha',\beta')| + |\Psi_{2,t}(\alpha,\beta) -  \Psi_{2,t}(\alpha',\beta')| \leq \frac{\kappa}{t^{{5}/{6}}} \bigg(\int_0^t |\a_s -\a'_s| \dd s + \int_0^t |\b_s -\b'_s| \dd s\bigg) , 
\end{equation}
which for any $t\in[0,T],$ implies that
\begin{align}
\ee^{-\lambda t}  |\Psi_{1,t}(\alpha,\beta) -  \Psi_{1,t}(\alpha',\beta')| &\leq \frac{\kappa}{t^{5/6}} \bigg(\int_0^t \ee^{-\lambda (t-s)} \ee^{-\lambda s} |\a_s -\a'_s| \dd s + \int_0^t \ee^{-\lambda (t-s)} \ee^{-\lambda s}|\b_s -\b'_s| \dd s\bigg)\\
&\leq  \big( \|  \a-\a' \|_{T,\lambda} + \|  \b-\b' \|_{T,\lambda} \big) \frac{\kappa}{\lambda^{1/6}} \frac{(1-\ee^{-\lambda t})}{(\lambda t)^{5/6}}\\
&\leq \big( \|  \a-\a' \|_{T,\lambda} + \|  \b-\b' \|_{T,\lambda} \big) \frac{\kappa}{\lambda^{1/6}}.
\end{align}
The same estimate holds for $ |\Psi_{2,t}(\alpha,\beta) -  \Psi_{2,t}(\alpha',\beta')|$ and thus, taking $\lambda$ suitably large yields the result.
\end{proof}

\subsection{The case $a,b\notin L^{1}(\mathbb{R}^d)$}
\label{sec:damping}
In order to extend the previous proposition to the case where $a,b$ are not necessarily integrable, we make use of the so-called \emph{damping method}. The intuitive idea behind it is rather simple and can be summarized, loosely, as follows. Assume that we wish to compute the expectation $\Eb[g(Z)]$ for a given function $g$ that has no Fourier transform (say, $g\in L^{\infty}$), and for random variable $Z$ whose density $f_Z$ is fast decreasing in the tails. Then, we seek a \emph{damping function} $\varphi$ such that both $\overline{g}(x):=\varphi(x)g(x)$ and $f_Z/\varphi$ admit a Fourier transform. The inversion formula then yields
\begin{equation}
\int_{\R} g(x)f_Z(x)\dd x =  \frac{1}{2\pi} \int_{\R}   \hat{\overline{g}}(\eta)\, \mathcal{F}(f_Z/\varphi)(-\eta) \dd \eta,
\end{equation}
which is useful as long as we have an explicit expression for $\mathcal{F}(f_Z/\varphi)(-\eta)$. In our case, we choose a damping function for $a$ and $b$ of the type $\varphi(x)=(1+\sum_{j=1}^d x_j^{q})^{-1}$, for a suitable even $q\in\mathbb{N}$ such that $a \varphi,b \varphi \in L^1(\mathbb{R}^d)\cap L^2(\mathbb{R}^d)$. This choice allows us to take advantage of the Fourier transform properties and compute their Fourier transform as
\begin{equation}
\mathcal{F}(f_Z/\varphi)(-\eta) = \Big(1+i^q \sum_{j=1}^d \partial_{\eta_j}^{q}\Big) \hat{f_Z}(-\eta) .
\end{equation}

Hereafter, we set $q= 2\ceil*{({d+1})/{2}}$ (the smallest positive even integer greater or equal than $d+1$) and define: 
\begin{equation}\label{eq:ste28}
\overline{a}(x):= \frac{a(x)}{1+\sum_{j=1}^d x_j^q}, \qquad \overline{b}(x):= \frac{b(x)}{1+\sum_{j=1}^d x_j^q},\qquad x\in\mathbb{R}^d.
\end{equation}
We note that, for this choice of $q$, the functions $\overline{a},\overline{b}\in L^1(\mathbb{R}^d)\cap L^2(\mathbb{R}^d)$ under the assumption that $a,b\in L^{\infty}(\mathbb{R}^d)$. We are then able to weaken Assumption \ref{assump:coeff} as following:  
\begin{assumption}\label{assump:coeff_bis}
The coefficients $a,b\in L^{\infty}(\R^d)$.
\end{assumption}
To replace Assumption \ref{assump:coeff} with Assumption \ref{assump:coeff_bis}, we pay the following cost which is an additional condition on the L\'evy measure $\nu(\dd y)$ and distribution of the initial datum $Y$, in order to ensure that the function
$x\mapsto\big(1+\sum_{j=1}^d x_j^q\big) f_{X^{(\alpha,\beta)}_t}(x)$ belongs to $L^{2}(\mathbb{R}^d)$.
\begin{assumption}
\label{ass:levy_bis}
For $q= 2\ceil*{({d+1})/{2}},$ the L\'evy measure $\nu$ is such that
\begin{equation}
\label{eq:ste52_bis}
\bar{n}_q:= \int_{|y|\geq 1}|y|^{q+1} \nu(\dd y) <\infty,
\end{equation}
and the initial datum $Y$ is such that 
\begin{equation}
\label{eq:finite_moments}
\Eb[|Y|^{q+1}]<\infty.
\end{equation}
\end{assumption}

\begin{proposition}\label{prop_verification_bis}[Fourier representation]
Under Assumption \ref{ellipticity}, \ref{assump:coeff_bis} and \ref{ass:levy_bis}, for any $T>0$ and $(\a,\b)\in L^{\infty}([0,T]:\mathcal{M}^{d\times d}\times \R^d), t\in[0,T],$ we have that
\begin{align}\label{eq:fourier_rep1_bis}
{\Psi_{1,t}}(\alpha,\beta)&= \frac{1}{(2\pi)^d} \int_{\R^d} \hat{\bar{a}}(\eta)\, \Lc\bigg(\!\! \exp\bigg(\!{-\frac{1}{2}\langle \eta, C^{\alpha}_t\eta\rangle - i\,\langle  \eta, m^{\alpha,\beta}_t\rangle +   n^{\alpha}_t(-\eta)}\!\bigg) 
\hat{\mu}_Y \!\left(-(\Phi^{\alpha}_{0,t})^\top \eta\right)\!\!\bigg)\dd \eta,\\ \label{eq:fourier_rep2_bis}
{\Psi_{2,t}}(\alpha,\beta)&= \frac{1}{(2\pi)^d} \int_{\R^d} \hat{\bar{b}}(\eta)\, \Lc\bigg(\! \exp\bigg(\!{-\frac{1}{2}\langle \eta, C^{\alpha}_t\eta\rangle - i\,\langle  \eta, m^{\alpha,\beta}_t\rangle +   n^{\alpha}_t(-\eta)}\!\bigg) 
\hat{\mu}_Y \!\left(-(\Phi^{\alpha}_{0,t})^\top \eta\right)\!\!\bigg)\dd \eta,
\end{align}
for any $t\in[0,T]$, where $C^{\alpha},m^{\alpha,\beta},n^{\alpha},\Phi^{\alpha}$ are defined as in \eqref{eq:ste10}--\eqref{eq:ode}, and $\Lc$ is the operator defined as
\begin{equation}\label{eq:q}
\Lc :=  \Big(1+i^q \sum_{j=1}^d \partial_{\eta_j}^{q}\Big),\quad q:= 2\ceil*{({d+1})/{2}}.
\end{equation}
\end{proposition}
\begin{remark}
Note that the damping functions in \eqref{eq:ste28} are not the only possible choices. For instance, choosing 
\begin{equation}\label{eq:ste28_bis}
\overline{a}(x):= \frac{a(x)}{\prod_{j=1}^d (1+ x_j^2)}, \qquad \overline{b}(x):= \frac{b(x)}{\prod_{j=1}^d (1+ x_j^2)},\qquad x\in\mathbb{R}^d,
\end{equation}
and slightly reinforcing Assumption \ref{ass:levy_bis}, Proposition \ref{prop_verification_bis} would still hold true with $\Lc =   \prod_{j=1}^d (1+\partial_{\eta_j}^{2})$. This choice could be more suitable, sometimes, in order to explicitly compute $\hat{\bar{a}}(\eta)$ and $\hat{\bar{b}}(\eta)$.
\end{remark}
\begin{proof}
For any $j=1,\cdots, d,$ by employing \eqref{eq:est_1}, \eqref{eq:ste52_bis}, and the dominated convergence theorem, we obtain that
\begin{align}\label{eq:der_n_first}
\partial_{\eta_j} n^{\alpha}_t(\eta)& = i \int_{0}^t  \bigg( \int_{|y|<1}  \big( \Phi^{\alpha}_{s,t} y\big)_j \Big( \ee^{i \langle  \eta , \Phi^{\alpha}_{s,t} y\rangle} -1\Big) \nu(\dd y) + \int_{|y|\geq 1}  \big( \Phi^{\alpha}_{s,t} y\big)_j  \ee^{i \langle  \eta , \Phi^{\alpha}_{s,t} y\rangle}  \nu(\dd y)\bigg)  \dd s,\\
\partial^{m}_{\eta_j} n^{\alpha}_t(\eta) &= i^{m}  \int_{0}^t  \int_{\R^d}   \big( \Phi^{\alpha}_{s,t} y\big)_j^m  \ee^{i \langle  \eta , \Phi^{\alpha}_{s,t} y\rangle}  \nu(\dd y)  \dd s,\qquad m=2,\cdots, q, \label{eq:der_n_m}
\end{align}
and
\begin{align}\label{eq:bound_der_n}
| \partial_{\eta_j} n^{\alpha}_t(\eta)|  &\leq  t \bigg( |\eta| \ee^{2 T \|\a \|_{T,0}}  \int_{|y|<1}  |y|^2  \nu(\dd y) +  \ee^{ T \|\a \|_{T,0}} \int_{|y|\geq 1}|y|\nu(\dd y) \bigg)   \\ \label{eq:bound_der_n_higher}
|\partial^{m}_{\eta_j} n^{\alpha}_t(\eta)|   &\leq      t   \ee^{m T \|\a \|_{T,0}}  \int_{\R^d}  |y|^m  \nu(\dd y), \qquad m=2,\cdots,q.
\end{align}
Moreover, for any $j=1,\cdots, d$ and $m=1,\cdots,q,$ by \eqref{eq:finite_moments} it is straightforward to see that 
\begin{equation}
\partial^m_{\eta_j} \hat{\mu}_Y \big(  (\Phi^{\alpha}_{0,t})^\top \eta\big) = i^m \int_{\R^d}  \big( \Phi^{\alpha}_{0,t} y\big)^m_j \ee^{i\langle \eta , \Phi^{\alpha}_{0,t} y \rangle} \mu_{Y}(\dd y),
\end{equation}
and
\begin{equation}\label{eq:der_charact_Y}
\big|\partial^m_{\eta_j} \hat{\mu}_Y \big(  (\Phi^{\alpha}_{0,t})^\top \eta\big) \big| \leq  \ee^{m T \|\a \|_{T,0}} E\big[  |Y|^m \big].
\end{equation}
Combining these estimates with \eqref{eq:ste13}, \eqref{eq:estim_n} and \eqref{eq:estimate_quad_form}, with $\delta$ suitably small, we conclude that 
\begin{equation}
\partial^m_{\eta_j} \hat{\mu}_{X^{(\a,\b)}_t}  \in L^p(\R^d), \quad  j=1,\cdots, d,\quad m=1,\cdots,q, \quad p\in \mathbb{N}.
\end{equation}
In particular, recalling that ${\mu}_{X^{(\a,\b)}_t}$ has a density $f_{X^{(\a,\b)}_t}$ given by \eqref{eq:density}, basic properties of the Fourier transform on $L^2(\R^d)$ yield that
\begin{equation}
\mathcal{F}^{-1} \big( i^m \partial^m_{\eta_j} \hat{\mu}_{X^{(\a,\b)}_t}\big) (x) =x_j^m {f}_{X^{(\a,\b)}_t} (x) \in L^2(\R^d).
\end{equation}
Assumption \ref{assump:coeff_bis} and \eqref{eq:ste28} also yield that $\bar{a},\bar{b}\in L^{1}(\R^d)\cap L^{2}(\R^d)$. Summing up, we can apply Plancherel's theorem and obtain the following
\eqstar{
\Psi_{1,t}(\alpha,\beta) = \int_{\R^d} a(x) {f}_{X^{(\a,\b)}_t}(x) \dd x &= \int_{\R^d} \bar{a}(x) 
\Big( 1+\sum_{j=1}^d x_j^q \Big) {f}_{X^{(\a,\b)}_t}(x) \dd x\\
&= \frac{1}{(2\pi)^d} \int_{\R^d}   \hat{\bar{a}}(\eta) \, \Lc \hat{\mu}_{X^{(\a,\b)}_t} (-\eta)  \dd \eta,
}
which, combined with \eqref{eq:ste13}, proves \eqref{eq:fourier_rep1_bis}. The proof of \eqref{eq:fourier_rep2_bis} is identical. 
\end{proof}

\begin{assumption}\label{assump:coeff_bis_B}
The Fourier transforms $\hat{\overline{a}},\hat{\overline{b}}\in L^1 (\R^d)$.
\end{assumption}

\begin{remark}
In order to prove the contraction property of the map $\Psi$ in the Fourier space, we were able to relax the assumption that $a,b\in L^1 (\R^d).$ However, the assumption that $\hat{\overline{a}},\hat{\overline{b}}\in L^1 (\R^d)$ cannot be relaxed except for the case $d=1$. It is necessary in order to handle general L\'evy jumps in the dynamics. For only diffusive dynamics, i.e.\ $L_t = \sigma W_t$ in \eqref{eq:levy_proc}, the contraction property, and thus, the existence and uniqueness for the fixed-point equation \eqref{eq:ste15}, can be proved by working in the original space under the sole assumption that $a,b \in L^{\infty}(\R^d)$. Also note that Proposition \ref{prop_verification_bis} does not rely on $\hat{\overline{a}},\hat{\overline{b}}\in L^1 (\R^d)$, and thus the Fourier representation \eqref{eq:fourier_rep1_bis}-\eqref{eq:fourier_rep2_bis} can be used for computational purposes as long as $a,b\in L^{\infty}(\R^d)$.
\end{remark}

\begin{theorem}\label{lem:contraction_bis}[Contraction property]
Suppose that Assumption \ref{ellipticity}, \ref{assump:coeff_bis}, \ref{ass:levy_bis} and \ref{assump:coeff_bis_B} hold. Then, for any $T,c>0$, there exists $\lambda>0$, only dependent on $c$, $T$, $\theta$, $\| a \|_{L^{\infty}}$, $\| b \|_{L^{\infty}}$,$\| \hat{\bar{a}} \|_{L^1}$, $\| \hat{\bar{b}} \|_{L^1}$, $
\nu(\dd y)$, $\Eb[|Y|^q]$ (with $q$ as in \eqref{eq:q}), and dimension $d$, such that 
\begin{align}\label{eq:contraction_bis}
\|\Psi(\alpha,\beta) - \Psi(\alpha',\beta')\|_{T,\lambda}& \leq c   \|(\alpha -\alpha',\beta -\beta' )\|_{T,\lambda},
\end{align}
for any $(\alpha,\beta),(\alpha',\beta')\in L^{\infty}([0,T]:\mathcal{M}^{d\times d}\times \R^d)$ with $\| (\alpha,\beta)\|_{T,0} ,\| (\alpha',\beta')\|_{T,0} \leq   \| a \|_{L^{\infty}}+ \| b \|_{L^{\infty}}$.
\end{theorem}

In the one-dimensional case the assumption $\hat{a},\hat{b}\in L^1 (\R^d)$ is not necessary. 
\begin{theorem}\label{lem:contraction_bis_d1}
Suppose that Assumption \ref{ellipticity}, \ref{assump:coeff_bis} and \ref{ass:levy_bis} hold, and $d=1$. Then the same conclusion as of Theorem \ref{lem:contraction_bis} holds, with $\lambda$ that depends on $\| \hat{\bar{a}} \|_{L^2}$, $\| \hat{\bar{b}} \|_{L^2}$ instead of $\| \hat{\bar{a}} \|_{L^1}$, $\| \hat{\bar{b}} \|_{L^1}$.
\end{theorem}
\begin{corollary}\label{cor:fixed_point_bis}
Under the assumptions of Theorem \ref{lem:contraction_bis} (or Theorem \ref{lem:contraction_bis_d1}), the same conclusion as of Corollary \ref{cor:fixed_point} holds.
\end{corollary}

The proofs of Theorem \ref{lem:contraction_bis_d1} and Corollary \ref{cor:fixed_point_bis} are identical to the proofs of their counterparts Theorem \ref{lem:contraction_d1} and Corollary \ref{cor:fixed_point}, respectively. The proof of Theorem \ref{lem:contraction_bis} is preceded by the following lemma on a priori estimates whose proof is relegated to Appendix \ref{app:estimates_apriori}.
\begin{lemma}\label{lemm:estim_der_n}
Suppose that Assumption \ref{assump:coeff_bis} and \ref{ass:levy_bis} hold. For any $T>0$, and $\alpha,\alpha'\in L^{\infty}([0,T]:\mathcal{M}^{d\times d})$ with $\|\alpha\|_{T,0},\|\alpha'\|_{T,0}\leq \| a \|_{L^{\infty}(\R^d)}$, it holds that
\begin{align}
\label{eq:ste104quat_bis}
& \big|\partial^{m}_{\eta_j}  n^{\alpha}_t(\eta) - \partial^{m}_{\eta_j}  n^{\alpha'}_t(\eta)  \big| \leq   \kappa\, t (|\eta|+1) \int_0^t | \alpha_u - \alpha'_u | \dd u ,\qquad j=1,\cdots,d,\quad  m=1,\cdots, q,
\end{align}
for any  $\eta\in\R^d, 0\leq  t\leq T$,
where $\kappa$ is a positive constant that depends only on $T,\| a \|_{L^{\infty}}$, and $\nu(\dd y)$. 
\end{lemma}
\begin{proof}[Poof of Theorem \ref{lem:contraction_bis}.]
Throughout this proof, we denote by $\kappa$ any positive constant that depends at most on  $T$, $\theta$, $\| a \|_{L^{\infty}}$, $\| b \|_{L^{\infty}}$,$\| \hat{\bar{a}} \|_{L^1}$, $\| \hat{\bar{b}} \|_{L^1}$, $
\nu(\dd y)$, $\Eb[|Y|^q]$, and dimension $d$. It is enough to prove that 
\begin{align}
&|\Psi_{1,t}(\alpha,\beta) -  \Psi_{1,t}(\alpha',\beta')|+|\Psi_{2,t}(\alpha,\beta) -  \Psi_{2,t}(\alpha',\beta')|\\
&\leq \frac{\kappa}{\sqrt{t}} \, \bigg(\int_0^t |\a_s -\a'_s| \dd s + \int_0^t |\b_s -\b'_s| \dd s\bigg) \biggl(\frac{1}{(2\pi)^d}\int_{\R^d} \big(|\hat{\bar{a}}(\eta)| +|\hat{\bar{b}}(\eta)| \big) \dd \eta \biggr),
\label{eq:estim_Psi_bis}
\end{align}
for any $0\leq t\leq T$. With \eqref{eq:estim_Psi_bis} at hand, the proof can be concluded exactly like the proof of Theorem \ref{lem:contraction}.

By Proposition \ref{prop_verification_bis} together with triangular inequality, we obtain that 
\begin{equation}
|\Psi_{1,t}(\alpha,\beta) -  \Psi_{1,t}(\alpha',\beta')| \leq  \frac{1}{(2\pi)^d} \bigg(\sum_{i=1}^5 I_i + \sum_{i=1}^5  J_{i} \bigg),
\end{equation}
where  
\begin{align}
J_1 &=     \int_{\R^d} |\hat{\bar{a}}(\eta)| \left| \partial_{\eta}^q \Big( \Big( \ee^{-\frac{1}{2}\langle \eta, C^{\alpha}_t\eta\rangle} - \ee^{-\frac{1}{2}\langle \eta, C^{\alpha'}_t\eta\rangle } \Big)  \ee^{- i\,\langle  \eta, m^{\alpha,\beta}_t\rangle +   n^{\alpha}_t(-\eta) }  \hat{\mu}_Y \big(\!-\!(\Phi_{0,t}^{\alpha})^\top \eta\big)\Big)\right| \dd \eta     ,\\
J_2 &=   \int_{\R^d} |\hat{\bar{a}}(\eta)| \left|  \partial_{\eta}^q\Big( \ee^{ -\frac{1}{2} \langle \eta, C^{\alpha'}_t \eta \rangle + n^{\alpha}_t (-\eta) }\Big( \ee^{- i\,\langle  \eta, m^{\alpha,\beta}_t\rangle } - \ee^{- i\,\langle  \eta, m^{\alpha',\beta}_t\rangle } \Big)  \hat{\mu}_Y \big(\!-\!(\Phi_{0,t}^{\alpha})^\top \eta\big)\Big)\right| \dd \eta     ,\\
J_3 &=   \int_{\R^d} |\hat{\bar{a}}(\eta)| \left|  \partial_{\eta}^q\Big( \ee^{ -\frac{1}{2} \langle \eta, C^{\alpha'}_t \eta \rangle + n^{\alpha}_t (-\eta) }\Big( \ee^{- i\,\langle  \eta, m^{\alpha',\beta}_t\rangle } - \ee^{- i\,\langle  \eta, m^{\alpha',\beta'}_t\rangle } \Big) \hat{\mu}_Y \big(\!-\!(\Phi_{0,t}^{\alpha})^\top \eta\big)\Big)\right| \dd \eta     ,\\
J_4 &=    \int_{\R^d} |\hat{\bar{a}}(\eta)| \left|  \partial_{\eta}^q\Big( \ee^{ -\frac{1}{2} \langle \eta, C^{\alpha'}_t \eta \rangle  - i\,\langle  \eta, m^{\alpha',\beta'}_t\rangle } \Big( \ee^{n^{\alpha}_t (-\eta)  } - \ee^{ n^{\alpha'}_t (-\eta) } \Big)  \hat{\mu}_Y \big(\!-\!(\Phi_{0,t}^{\alpha})^\top \eta\big)\Big)\right| \dd \eta     ,\\
J_5 &=    \int_{\R^d} |\hat{\bar{a}}(\eta)| \left|  \partial_{\eta}^q\Big( \ee^{ -\frac{1}{2} \langle \eta, C^{\alpha'}_t \eta \rangle  - i\,\langle  \eta, m^{\alpha',\beta'}_t\rangle + n^{\alpha'}_t (-\eta) } \Big( \hat{\mu}_Y  \big(\!-\!(\Phi_{0,t}^{\alpha})^\top \eta\big) - \hat{\mu}_Y  \big(\!-\!(\Phi_{0,t}^{\alpha'})^\top \eta\big) \Big) \Big)\right| \dd \eta.
\end{align}
The terms $I_i$ are like those in the proof of Theorem \ref{lem:contraction} and are obtained by replacing $\hat{a}$ with $\hat{\bar{a}}$, and can be bounded in the same way. Analogous bounds for $J_i$ can be obtained by repeatedly applying the estimates of Lemma \ref{lemm:estimates_bis}, \ref{lem:estimates} and \ref{lemm:estim_der_n}, and the estimates \eqref{eq:bound_der_n}-\eqref{eq:bound_der_n_higher}-\eqref{eq:der_charact_Y}. Eventually, applying \eqref{eq:guassian_estimate} yields \eqref{eq:estim_Psi_bis}. We omit the details to avoid repeating the arguments from the proof of Theorem  \ref{lem:contraction}. 
\end{proof}
 
\section{Discretized Picard iteration scheme}
\label{section: discrete Picard}
In principle, we can compute the unique solution $(\bar{\a},\bar{\beta})$ to the MK-V fixed-point equation \eqref{eq:ste15} via a Picard iteration scheme. The result of Corollary \ref{cor:fixed_point} (or Corollary \ref{cor:fixed_point_bis}) together with the Fourier representation \eqref{eq:fourier_rep1}-\eqref{eq:fourier_rep2} (or \eqref{eq:fourier_rep1_bis}-\eqref{eq:fourier_rep2_bis}) provide us with a convergent scheme. However, for a given initial point $\gamma^0 = (\alpha^0,\beta^0) \in L^{\infty}([0,T]:\mathcal{M}^{d\times d}\times \R^d)$, the approximating sequence 
\begin{equation}\label{eq:continuous_picard}
\gamma^{m}= (\alpha^m,\beta^m):= \Psi(\gamma^{m-1})=\Psi(\alpha^{m-1},\beta^{m-1}),\quad m\in\mathbb{N}, 
\end{equation}
cannot be computed explicitly at each step. Although, we do not consider here the effect of the error introduced by numerically approximating the space integral in \eqref{eq:fourier_rep1}-\eqref{eq:fourier_rep2} (or \eqref{eq:fourier_rep1_bis}-\eqref{eq:fourier_rep2_bis}), we do analyse the impact of time-discretization of functions $\Phi^{\alpha,\beta}_{s,t}$, $C^{\alpha}_t$, $m^{\alpha,\beta}_t$ on the convergence rate of the numerical scheme. 

For any $\g=(\a,\b)\in L^{\infty}([0,T]:\mathcal{M}^{d\times d}\times \R^d)$ and $n\in\mathbb{N}$, we define the piece-wise constant function $\g^{(n)}$ as
$$
\g^{(n)}_t:= \sum_{i=1}^{n} \g_{t_i} {\bf 1}_{[t_{i-1},t_{i}[}(t), \quad t\in[0,T], \quad \text{where } t_i: = \frac{T}{n} i .
$$
We also set operator $\Psi^{(n)}$ from $L^{\infty}([0,T]:\mathcal{M}^{d\times d}\times \R^d)$ onto itself acting as
$$
\Psi^{(n)}(\gamma) := \big(\Psi(\gamma)\big)^{(n)}, \quad \gamma \in L^{\infty}([0,T]:\mathcal{M}^{d\times d}\times \R^d).
$$
The map $\Psi^{(n)}$ has to be interpreted as a step-wise approximation of $\Psi$, and is the map we compute in our Picard iteration scheme. The idea is to repeatedly apply operator $\Psi^{(n)}$ instead of $\Psi$, in order to take advantage of the fact that  
 $\Phi^{\alpha,\beta}_{s,t}$, $C^{\alpha}_t$, $m^{\alpha,\beta}_t$ are explicitly computable if $\a$ and $\b$ are step-functions. Precisely, 
for a given initial step-function $\gamma^{0,n} \in L^{\infty}([0,T]:\mathcal{M}^{d\times d}\times \R^d)$, the approximating sequence 
\begin{equation}\label{eq:discrete_picard}
 \gamma^{m,n}=(\alpha^{m,n},\beta^{m,n}) := \Psi^{(n)}(\gamma^{m-1,n})= \Psi^{(n)}(\alpha^{m-1,n},\beta^{m-1,n}),\quad m\in\mathbb{N}, 
\end{equation}
can be computed explicitly at each step, up to computing a space integral on $\R^d.$ This is due to the fact that the solution $\Phi^{\alpha}$ to ODE \eqref{eq:ode} can be computed explicitly in terms of matrix exponentials whenever $\alpha$ is a piece-wise constant.
Hereafter, we assume that the continuous and discretized Picard iterations, defined by \eqref{eq:continuous_picard} and \eqref{eq:discrete_picard} respectively, are both initialized by the same constant function, i.e.,
\begin{equation}\label{eq:initialization}
\gamma^{0}_t=\gamma^{0,n}_t \equiv {\gamma}_{0}=(\a_0,\b_0)\in \mathcal{M}^{d\times d}\times \R^d, \quad t\in [0,T].
\end{equation}
Note that ${\g}^{m,n} \neq ({\g}^{m})^{(n)}$ which means that ${\g}^{m,n}$ is not the discretized version of ${\g}^{m}.$ 

In order to be able to control the error introduced by the time-discretization, we must be able to study the regularity of the function $t\mapsto \Psi_t(\gamma) $ on $[0,T]$. For this purpose, we need to introduce some further assumptions on coefficients $a,b$ and/or on distribution of the initial datum $Y$, which are needed to ensure Lipschitz continuity of $\Psi_t(\gamma)$ near $t=0$.

\begin{assumption}
\label{ass:additional_assumption}
$\hat{a}, \hat{b}$ and $\hat{\mu}_Y$ satisfy the following conditions
\begin{align}
\label{eq:additional_assumption}
\int_{\R^d}  \big(  |\hat{a}(\eta)|  + |\hat{b}(\eta)|    \big)|\hat{\mu}_Y (\eta)|   |\eta|^{2}   \dd \eta <\infty,     &&     \int_{\R^d}  \big(  |\hat{a}(\eta)|  + |\hat{b}(\eta)|  \big)  |\eta|  \dd \eta <\infty. 
\end{align}
\end{assumption}
\begin{remark}
Note that Assumption \ref{ass:additional_assumption} is related to regularity of the functions $a$, $b$ and of the distribution of the initial datum $Y$. For instance, the second condition in \eqref{eq:additional_assumption} is equivalent to requiring that $ \hat{a}(\eta)  |\eta|$ and $ \hat{b}(\eta) |\eta|$ belong to $L^1(\R^d)$, which implies that $a,b$ are continuously differentiable. Also, the first condition in \eqref{eq:additional_assumption} is satisfied if either the functions $\hat{a}(\eta)  |\eta|^2$, $ \hat{b}(\eta) |\eta|^2$, or the function $\hat{\mu}_Y (\eta) |\eta|^2$ belong to $L^1(\R^d)$, which in turn is satisfied if the coefficients $a,b$, or density of the initial datum $Y$ are $d+3$ times continuously differentiable with derivatives in $L^1(\R^d)$. Alternatively, the first condition in \eqref{eq:additional_assumption} is also ensured if $\hat{a}(\eta)  |\eta|$, $ \hat{b}(\eta) |\eta|$ and $\hat{\mu}_Y (\eta) |\eta|$ all belong to $L^2(\R^d)$, which in turn is ensured by requiring $a,b$ and the density of $Y$ belong to the first-order Sobolev space $H^1(\R^d)$.
All these conditions seem rather strong, but we claim that they are not necessary in many particular cases. Once again we emphasise that this is the cost we incur for carrying out the analysis in the Fourier space, which enables us to deal with general L\'evy measures.
\end{remark}
When working under the assumptions of Section \ref{sec:damping} ($a,b$ only in $L^{\infty}(\R^d)$), we will need instead the following additional assumptions.
\begin{assumption}\label{ass:additional_assumption_bis}
$\hat{\bar{a}},\hat{\bar{b}}$ and $\hat{\mu}_Y$ satisfy the following conditions\begin{align}\label{eq:additional_assumption_bis}
      \int_{\R^d}  \big(  |\hat{\bar{a}}(\eta)|  + |\hat{\bar{b}}(\eta)|    \big)|\hat{\mu}_Y (\eta)|   |\eta|^2   \dd \eta <\infty,     &&     \int_{\R^d}  \big(  |\hat{\bar{a}}(\eta)|  + |\hat{\bar{b}}(\eta)|    \big)  |\eta|   \dd \eta <\infty    ,
   \end{align}
   where $\bar{a}$ and $\bar{b}$ are as defined in \eqref{eq:ste28}.
   \end{assumption}
\begin{remark}
Note that Assumption \ref{ass:additional_assumption} and \ref{ass:additional_assumption_bis} imply Assumption \ref{assump:coeff_B} and \ref{assump:coeff_bis_B}, respectively.
\end{remark}
We now state the two main results of this section. 

\begin{theorem}
\label{th:main}
Let $(\gamma^{m,n})_{m,n\in\mathbb{N}}$ be the sequence as defined by \eqref{eq:discrete_picard}-\eqref{eq:initialization}. 
Suppose that  Assumption \ref{assump:coeff}, \ref{ellipticity}, \ref{ass:levy} and \ref{ass:additional_assumption} also hold. For any $T>0,{\g}_{0}=\big({\a}_{0},{\b}_{0}\big)\in  \mathcal{M}^{d\times d}\times \R^d$ with $\max \{ |{\a}_{0}| , |{\b}_{0} |  \} \leq \| a \|_{L^{\infty}
}+\| b \|_{L^{\infty}
},$ there exist $\lambda,\kappa>0$, only dependent on $T$, $\theta$, 
$\nu(\dd y)$, $Y$ and the coefficients $a,b$, such that
\begin{equation}\label{eq:convergence_scheme}
\|  \overline{\g} -{\g}^{m,n}  \|_{T,\lambda} \leq \kappa \Big(\frac{1}{2^m}+\frac{1}{n} \Big),\qquad n,m \in \mathbb{N},
\end{equation}
where $\overline{\gamma}=(\overline{\a}$, $\overline{\b})$ is the unique solution in $L^{\infty}([0,T]:\mathcal{M}^{d\times d}\times \R^d)$ to McKean-Vlasov fixed-point equation \eqref{eq:ste15}.
\end{theorem}
In analogy with the results of Section \ref{sec:damping}, for coefficients $a,b$ that are not in $L^1(\R^d)$ we have the following extension. 
\begin{theorem}\label{th:main_bis}
Under Assumption \ref{assump:coeff_bis}, \ref{ellipticity}, \ref{ass:levy_bis} and \ref{ass:additional_assumption_bis}, the same result as of Theorem \ref{th:main} holds.
\end{theorem}

The remaining part of the section is devoted to prove Theorem \ref{th:main}. The proof of Theorem \ref{th:main_bis} is analogous and thus, is omitted. {\bf From now on, through the rest of this section, we will assume that the hypotheses of Theorem \ref{th:main} are satisfied. In particular, we fix an arbitrary $T$, and a suitable $\lambda>0$ such that \eqref{eq:contraction} holds with $c=1/2$. Also, we will denote by $\kappa$ any positive constant that depends at most on $T$, $\theta$, $\nu(\dd y)$, $Y$ and coefficients $a,b$. Finally, we initialize the sequences $\gamma^{m,n}$ and $\gamma^{m}$ as in \eqref{eq:initialization} with $\gamma_0=(\a_0,\b_0)$ satisfying $\max \{ |{\a}_{0}| , |{\b}_{0} |  \} \leq \| a \|_{L^{\infty}
}+\| b \|_{L^{\infty}
} $.}

\begin{lemma}\label{lem:Lipsch_time}
For any $\g\in L^{\infty}([0,T]:\mathcal{M}^{d\times d}\times \R^d)$ with $\|\g\|_{T,0}\leq   \| a \|_{L^{\infty}}+ \| b \|_{L^{\infty}},$ we have the following
\begin{equation}\label{eq:lipschitz}
|  \Psi_{j,t} (\g) - \Psi_{j,t'} (\g)  | \leq \kappa |t-t'|, \qquad t,t'\in[0,T],\quad j=1,2,
\end{equation}
\end{lemma}
\begin{proof}
By Proposition \ref{prop_verification} and triangular inequality we obtain that
\begin{equation}
|\Psi_{1,t}(\alpha,\beta) -  \Psi_{1,t'}(\alpha,\beta)| \leq \frac{1}{(2\pi)^d} \sum_{i=1}^4 I_i,
\end{equation}
where
\begin{align}
I_1 &=     \int_{\R^d} |\hat{a}(\eta)| \left| \ee^{-\frac{1}{2}\langle \eta, C^{\alpha}_t\eta\rangle} - \ee^{-\frac{1}{2}\langle \eta, C^{\alpha}_{t'}\eta\rangle } \right| \left| \ee^{- i\,\langle  \eta, m^{\alpha,\beta}_t\rangle +   n^{\alpha}_t(-\eta) } \right| \big|\hat{\mu}_Y \big(\!-\!(\Phi_{0,t}^{\alpha})^\top \eta\big)\big| \dd \eta     ,\\
I_2 &=   \int_{\R^d} |\hat{a}(\eta)| \left| \ee^{ -\frac{1}{2} \langle \eta, C^{\alpha}_{t'} \eta \rangle + n^{\alpha}_t (-\eta) }\right| \left|\ee^{- i\,\langle  \eta, m^{\alpha,\beta}_t\rangle } - \ee^{- i\,\langle  \eta, m^{\alpha,\beta}_{t'}\rangle } \right|  \big|\hat{\mu}_Y \big(\!-\!(\Phi_{0,t}^{\alpha})^\top \eta\big)\big| \dd \eta     ,\\
I_3 &=    \int_{\R^d} |\hat{a}(\eta)| \left| \ee^{ -\frac{1}{2} \langle \eta, C^{\alpha}_{t'} \eta \rangle  - i\,\langle  \eta, m^{\alpha,\beta}_{t'}\rangle }\right| \left|\ee^{n^{\alpha}_t (-\eta)  } - \ee^{ n^{\alpha}_{t'} (-\eta) } \right|  \big|\hat{\mu}_Y \big(\!-\!(\Phi_{0,t}^{\alpha})^\top \eta\big)\big| \dd \eta     ,\\
I_4 &=    \int_{\R^d} |\hat{a}(\eta)| \left| \ee^{ -\frac{1}{2} \langle \eta, C^{\alpha}_{t'} \eta \rangle  - i\,\langle  \eta, m^{\alpha,\beta}_{t'}\rangle + n^{\alpha}_{t'} (-\eta) }\right| \big|\hat{\mu}_Y  \big(\!-\!(\Phi_{0,t}^{\alpha})^\top \eta\big) - \hat{\mu}_Y  \big(\!-\!(\Phi_{0,t'}^{\alpha})^\top \eta\big) \big| \dd \eta      .
\end{align}
We note that 
\begin{align}
\big|\hat{\mu}_Y \big(\!-\!(\Phi_{0,t}^{\alpha})^\top \eta\big)\big| & \leq |\hat{\mu}_Y ( \eta)| + \big|\hat{\mu}_Y \big(\!-\!(\Phi_{0,t}^{\alpha})^\top \eta\big) - \hat{\mu}_Y(\eta) \big|  \leq |\hat{\mu}_Y ( \eta)| + \big|\big(  (\Phi^{\alpha}_{0,t})^\top   -   (\Phi_{0,0}^{\alpha})^\top  \big)\eta  \big|  \, \E [|Y|]
\intertext{(by \eqref{eq:est_1_ter})}
& \leq |\hat{\mu}_Y ( \eta)| + \kappa t |\eta | . \label{eq:ste234}
\end{align}
By \eqref{eq:est_1_b_ter} and \eqref{eq:estimate_quad_form}, we obtain that
\begin{equation}
 \left| \ee^{-\frac{1}{2}\langle \eta, C^{\alpha}_t\eta\rangle} - \ee^{-\frac{1}{2}\langle \eta, C^{\alpha}_{t'}\eta\rangle } \right| \leq 
 \kappa\, |t - t' ||\eta|^2 \ee^{- \frac{t|\eta|^2}{2 \kappa}}  ,
\end{equation}
which, combined to \eqref{eq:estim_e_n} with $\eps$ suitably small, yields that
\begin{align}
I_1 &\leq\frac{ \kappa |t-t'| }{(2\pi)^d}\int_{\R^d}  |\hat{a}(\eta)|    \ee^{-\frac{ t|\eta|^2}{2 \kappa}} |\eta|^2 \big|\hat{\mu}_Y \big(\!-\!(\Phi_{0,t}^{\alpha})^\top \eta\big)\big| \dd \eta 
\intertext{(by \eqref{eq:ste234} along with \eqref{eq:guassian_estimate})}
&  \leq\frac{ \kappa |t-t'| }{(2\pi)^d} \int_{\R^d}  \ee^{-\frac{ t|\eta|^2}{2 \kappa}}  |\hat{a}(\eta)|  \big( |\eta|^2 \big|\hat{\mu}_Y(\eta)\big| + |\eta| \big) \dd \eta \leq \kappa |t-t'|.
\end{align}
In the last inequality above, we employed Assumption \ref{ass:additional_assumption}. Similarly, we find the same bound for $I_2$, $I_3$, $I_4$ by applying \eqref{eq:est_1_c_ter},  \eqref{eq:estim_n_ter} and \eqref{eq:estim_n}-\eqref{eq:estimate_quad_form} again. This proves \eqref{eq:lipschitz} for $j=1$. The proof for $j=2$ is identical. 
\end{proof}

\begin{lemma}
\label{lem:discrete map minus conts map}
For any $\g,\g' \in L^{\infty}([0,T]:\mathcal{M}^{d\times d}\times \R^d)$ with $\|\g\|_{T,0}\leq   \| a \|_{L^{\infty}}+ \| b \|_{L^{\infty}},$ we have that
\begin{equation}
\|  \Psi^{(n)}(\g') - \Psi({\g})  \|_{T,\lambda}\leq  \frac{\kappa}{n} + \frac{1}{2}\|  \g' - {\g}  \|_{T,\lambda}, \quad n\in\mathbb{N},
\end{equation}
\end{lemma}
\begin{proof}
By triangular inequality, we obtain that
\begin{equation}
\|  \Psi^{(n)}(\g') - \Psi( {\g})  \|_{T,\lambda}  \leq \|  \Psi^{(n)}(\g') - \Psi(\g')  \|_{T,\lambda} + \|  \Psi(\g') - \Psi({\g})  \|_{T,\lambda}.
\end{equation}
For the second term above, we use the contraction property of Theorem \ref{lem:contraction} with $c=1/2$. To see the bound for the first term, we write the following
\begin{align}
\big| \big(\Psi^{(n)}_{j,t}(\g')\big) - \big(\Psi_{j,t}(\g')\big) \big| \leq \sum_{i=1}^{n}\big| \big(\Psi_{j,t_i}(\g')\big) - \big(\Psi_{j,t}(\g')\big) \big|  {\bf 1}_{[t_{i-1},t_i[}, \quad t\in[0,T],\quad j=1,2,
\end{align}
which yields that
\begin{align}
\big\| \Psi_j^{(n)}(\g') - \Psi_j(\g')\big\|_{T,\lambda} &\leq \sup_{t\in [0,T]} \ee^{-\lambda t}\sum_{i=1}^{n}\big| \big(\Psi_{j,t_i}(\g')\big) - \big(\Psi_{j,t}(\g')\big) \big|  {\bf 1}_{[t_{i-1},t_i[}(t)
\intertext{(by applying Lemma \ref{lem:Lipsch_time})}
& \leq  \kappa \sup_{t\in [0,T]} \ee^{-\lambda t} \sum_{i=1}^{n} |t-t_i|   {\bf 1}_{[t_{i-1},t_i[}(t) \leq \frac{\kappa}{n},\quad t\in[0,T],\quad j=1,2.
\end{align}
This concludes the proof.
\end{proof}

\begin{remark}\label{rem:disc_cont}
For any $\g \in L^{\infty}([0,T]:\mathcal{M}^{d\times d}\times \R^d)$, it is easy to observe from the definition of $\Psi^{(n)}$ that 
\begin{equation}
\|  \Psi^{(n)}(\g) \|_{T,0}\leq \|  \Psi (\g)   \|_{T,0}.
\end{equation}
\end{remark}
\begin{lemma}\label{lem:conv_n}
For any $m,n\in\mathbb{N},$ we have that 
\begin{equation}\label{eq:estim_induction}
\|{\g}^{m,n} -{\g}^{m} \|_{T,\lambda} \leq 2\Big({1-\frac{1}{2^m}}\Big)\frac{\kappa}{n}\leq \frac{\kappa}{n}.
\end{equation}
\end{lemma}
\begin{proof}
First note that, by the assumptions on $\g^{0,n},\gamma^{n}$ and Remark \ref{rem:disc_cont}, we have that
\begin{equation}
\|{\g}^{m,n}\|_{T,0} , \|{\g}^{m}\|_{T,0} \leq  \| a \|_{L^{\infty}}+ \| b \|_{L^{\infty}}, \quad m,n\in \mathbb{N}\cup \{ 0 \}.
\end{equation}
We now prove the result by induction. For $m = 1$, by Lemma \ref{lem:discrete map minus conts map} and the fact that $\g^{0,n}=\gamma^{n}$, we obtain that
\eqstar{
\|{\g}^{1,n} -{\g}^{1} \|_{T,\lambda} = \| \Psi^{(n)}({\g}^{0,n}) - \Psi({\g}^{0}) \|_{T,\lambda} \leq \frac{\kappa}{n}.
}
Now suppose that \eqref{eq:estim_induction} holds for $m-1,$ and we prove it true for $m$. Again, applying Lemma \ref{lem:discrete map minus conts map} yields that
\eqstar{
\|{\g}^{m,n} -{\g}^{m} \|_{T,\lambda} &= \| \Psi^{(n)}({\g}^{m-1,n}) - \Psi({\g}^{m-1}) \|_{T,\lambda} \leq \frac{\kappa}{n} + \frac{1}{2}  \|{\g}^{m-1,n} -{\g}^{m-1} \|_{T,\lambda}
\intertext{(by inductive hypothesis)}
& \leq \frac{\kappa}{n} + \Big({1-\frac{1}{2^{m-1}}}\Big)\frac{\kappa}{n} = 2\Big({1-\frac{1}{2^m}}\Big)\frac{\kappa}{n}.
} 
\end{proof}
We are now in the position to prove Theorem \ref{th:main}.

\begin{proof}[Proof of Theorem \ref{th:main}] 
First note that, by assumption on $\g^0$ and by Corollary \ref{eq:set_B}, we have that 
\begin{equation}
\|  {\g}^{0} - \overline{\g} \|_{T,\lambda}\leq \|\g^{0}\|_{T,0}+\|\overline{\g}\|_{T,0}\leq  2 (\| a \|_{L^{\infty}}+\| b \|_{L^{\infty}}).
\end{equation}
Therefore, by applying the contraction property in Theorem \ref{lem:contraction} with $c=1/2$ it is straightforward to see that 
\begin{equation}\label{eq:converg_cont_picard}
 \|  {\g}^{m} - \overline{\g} \|_{T,\lambda} \leq 2^{-(m-1)}  (\| a \|_{L^{\infty}}+\| b \|_{L^{\infty}}) .
\end{equation}
 Now, by triangular inequality, we obtain that
\begin{equation}
\|{\g}^{m,n} - \overline{\g} \|_{T,\lambda} \leq \|{\g}^{m,n} -{\g}^{m} \|_{T,\lambda} + \| {\g}^{m} - \overline{\g} \|_{T,\lambda}.
\end{equation}
The result follows from applying Lemma \ref{lem:conv_n} to the first term, and \eqref{eq:converg_cont_picard} to second.
\end{proof}

\section{Numerical results}
\label{section:numerics}
In this section, we demonstrate the applicability of our theoretical results by testing them on examples for which semi-explicit solutions are available. We verify the convergence of the discretized Picard iteration scheme and the rate of error convergence as discussed in Section \ref{section: discrete Picard}. We performed all the numerical computations on a computing device with 2,4 GHz Intel i5 processor and 16 GB RAM. 

\subsection{Gaussian benchmark}
\label{example: gaussian}
We first consider the following SDE in one-dimension
\begin{equation}\label{eq:example_1D}
\dd X_t  = \bigl( a X_t +  \Eb[\cos(X_t)] \bigr) \dd t + \sigma d W_t  , \ t>0, \quad X_0 = Y,
\end{equation}
with $a\in \R$ and $\sigma>0$, and where $Y$ has Laplace distribution $\mu_Y(\dd y) = \frac{1}{2}\ee^{-|y|} \dd y$. Comparing with our setting of \eqref{eq:mckean_SDE_jumps}, it gives us that
\eqstar{
a(x) \equiv a, \quad b(x) =\cos(x),
}
and that $\nu(\dd y)\equiv 0$, i.e., there are no jumps. Here, $a$ and $b$ are bounded functions but they do not belong to $L^1(\R)$. However, considering the damped coefficients $\overline{a}(x) = \frac{a}{(1+x^2)}$ and $\overline{b}(x) = \frac{\cos x}{(1+x^2)}$, it is immediate to check that the Fourier transforms $\hat{\overline{a}}$, $\hat{\overline{b}}$ satisfy Assumption \ref{assump:coeff_bis_B} and \ref{ass:additional_assumption_bis}. Moreover, Assumption \ref{ellipticity} and \ref{ass:levy_bis} are also satisfied, and thus the results of 
 Sections \ref{sec:damping} and \ref{section: discrete Picard} do apply to \eqref{eq:example_1D}. In particular: we have the existence and uniqueness for the solution of the fixed-point equation \eqref{eq:ste15} from Corollary \ref{cor:fixed_point_bis}, and that the Picard sequence in \eqref{eq:discrete_picard}-\eqref{eq:initialization} converges to the solution $(\bar{\alpha},\bar{\beta})$ from Theorem \ref{th:main_bis}.
 
For this specification of $a,b$ it is possible to obtain semi-explicit solutions $(\bar{\a},\bar{\b})$ for \eqref{eq:ste15}, which we use as a benchmark to test the rate of convergence of our approximating scheme. In fact, by Lemma \ref{lem:ito} we have that $(\bar{\a},\bar{\b})$ satisfies
\begin{align}
\bar{\a}_t &= a,\\
\bar{\b}_t &= \Eb[\cos(X^{(\bar{\a},\bar{\b})}_t)] = 
\frac{\hat{\mu}_{X^{(\bar{\a},\bar{\b})}_t}(1) + \hat{\mu}_{X^{(\bar{\a},\bar{\b})}_t}(-1)}{2} = \frac{\ee^{-\frac{1}{2}C_t}}{1+ \ee^{2 a t}} \cos(m_t),
\end{align}
where we set

\begin{equation}\label{eq:m_C_True}
m_t := m^{\bar{\a},\bar{\b}}_t = \ee^{at} \int^t_0\ee^{- a s } \bX_s \dd s,\qquad C_t :=  C^{\bar{\a}}_t  =  \tfrac{\sigma^2}{2a}(\ee^{2at} - 1),
\end{equation}
respectively the mean and variance of the solution to \eqref{eq:example_1D}. We then obtain that $\bX_t = m^\prime_t - a m_t$ where $m_t$ solves the following equation
\eqlnostar{eq:ode m simple}{
m^\prime_t = a \, m_t + \frac{\ee^{-\frac{1}{2}C_t}}{1+ \ee^{2 a t}} \cos(m_t).
}
In the absence of a closed-form expression for $\bX_t,$ we treat the numerical solution from ODE \eqref{eq:ode m simple} as a proxy for the true value. We also point that, in this specific case, no numerical integration is required to compute $\gamma^{m,n}$, neither in the Fourier space nor in the original one. 

In order to verify the convergence rate of our method as derived in Theorem \ref{th:main}, we set the number of Picard iteration steps, $m = \log_2 (n)$ where $n$ is the number of time discretization steps. In Figure \ref{fig:error 1d simple}, for parameter values $a=1.5$, $\sigma = 0.8$, $T = 1.0,$ we vary $n = 2^k, 4 \leq k \leq 8$, and observe that the slope of log-error, i.e.,
\eqlnostar{eq:errror_mumerics}{
\log\Big( \max_{k=0,\cdots,n}{|\beta_{t_k}^{m,n} - \bar{\beta}_{t_k}|}   \Big) \quad (\text{with } t_k = \frac{k T}{n}), 
}
indeed matches the result in Theorem \ref{th:main}. In Figure \ref{fig:comparison 1d simple}, we compare the Picard scheme approximation $m^{\alpha^{m,n},\beta^{m,n}}_t$ of $m^{\bar{\a},\bar{\b}}_t$ with $n=2^4$ against the numerical solution obtained by solving ODE \eqref{eq:ode m simple}. This comparison shows that we obtain an accurate approximation even for small values of $n$ and $m$ in the discretized Picard iteration scheme.

\begin{figure}[htbp]
\centering
\subfigure[]{\label{fig:error 1d simple}\includegraphics[scale=0.45]{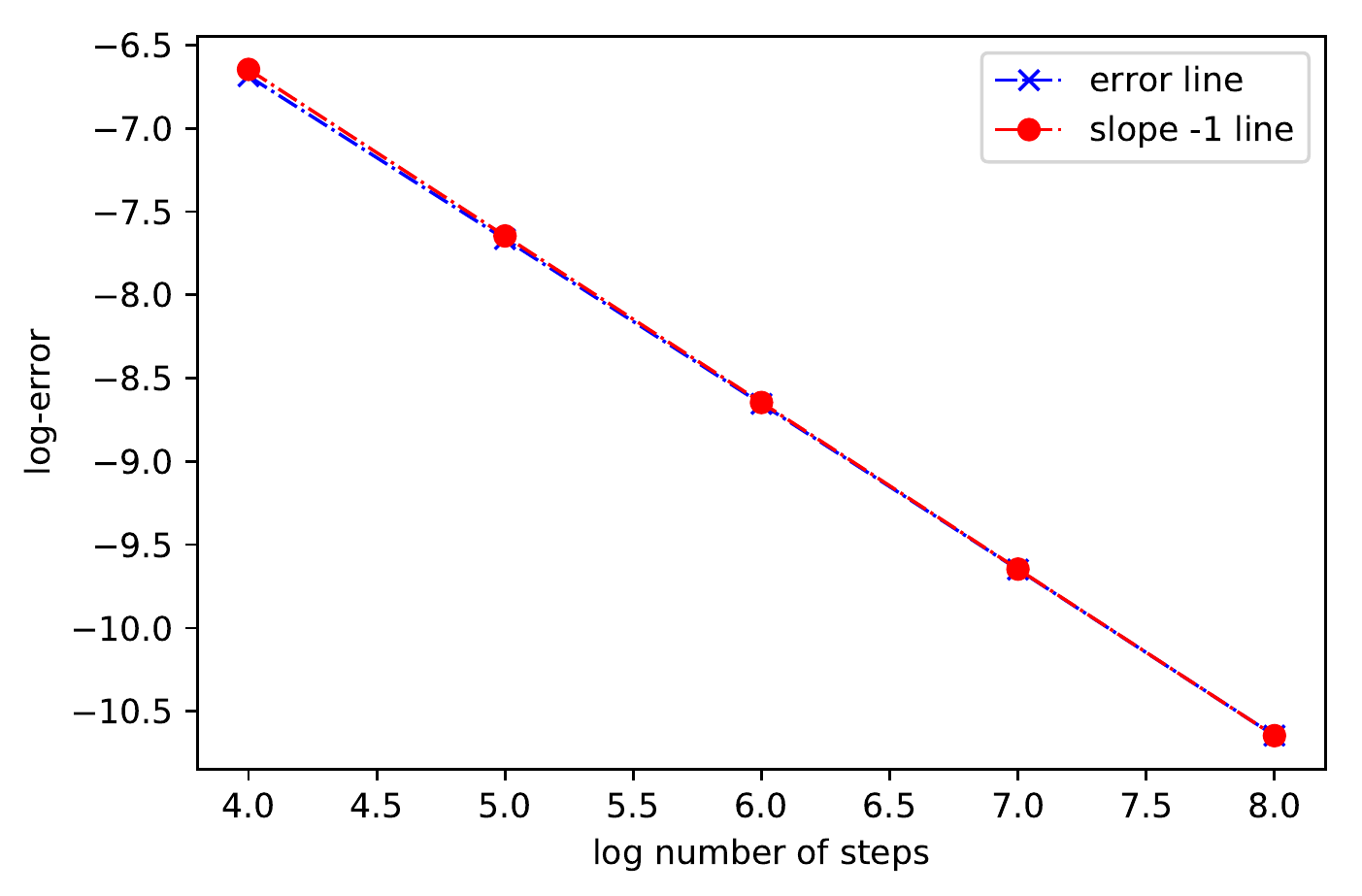}}
~
\subfigure[]{\label{fig:comparison 1d simple}\includegraphics[scale=0.45]{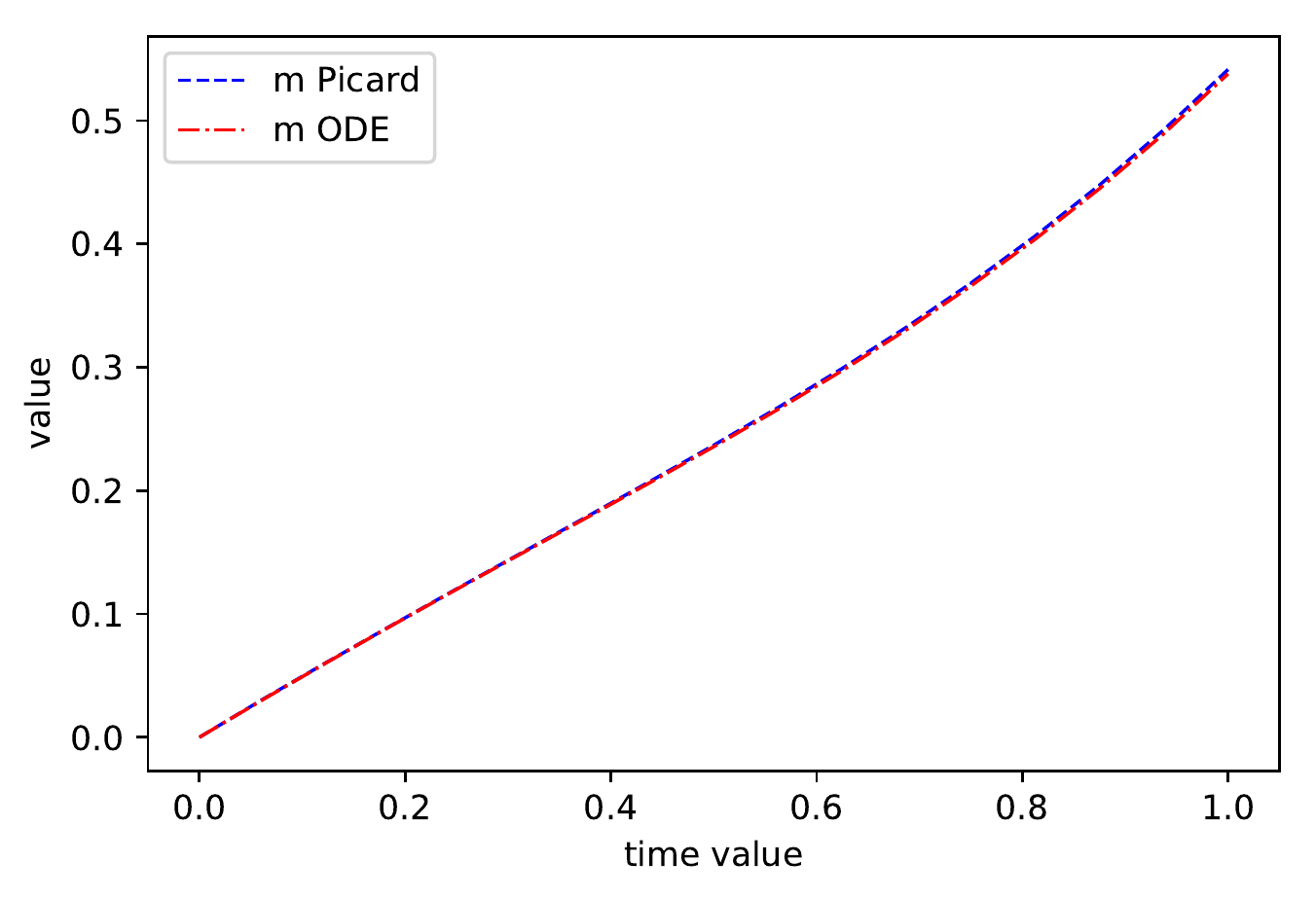}}\\
~
\caption{Error convergence rate and comparison of approximations for one-dimensional model.}
\label{fig:1d}
\end{figure}

\subsection{Jumps in one dimension}
\label{example:jumps 1d}
In this section, we generalize the previous example by adding L\'evy jumps to the McKean-Vlasov SDE in \eqref{eq:example_1D}. In particular, we consider a compound Poisson process with jump-intensity $\lambda$, where the distribution of jumps is defined in terms of an asymmetric double exponential density, i.e.,
\eqstar{
\chi(\dd y)  = \bigl(p \lambda_1 \ee^{-\lambda_1 y} \ind_{\{y > 0\}} + (1-p) \lambda_2 \ee^{-\lambda_2 y} \ind_{\{y < 0\}}\bigr) \dd y,
}
where $\lambda_1, \lambda_2 > 0$ and $p \in [0,1]$ represents the probability of upward jumps. The L\'evy measure of this process, which appears in some financial applications (see Kou model \cite{kou2002jump}), is $\nu = \lambda \chi$, and thus satisfies Assumption \ref{ass:levy_bis}. Therefore, the theoretical results of 
 Sections \ref{sec:damping} and \ref{section: discrete Picard} still apply to this case. We also note that the density of $\mu_{X^{(\a,\b)}_t}$ is not known in closed form and thus, our method based on Fourier transform has even more significance.
 
Since compound Poisson processes have finite activity, $\nu(\dd y)$ is a finite measure on $\R\setminus\{ 0 \}$. Thus it is convenient to simplify the L\'evy-Ito representation in \eqref{eq:levy_proc} by writing a pure-jump (non-compensated) stochastic integral on $\R$. This choice also simplifies the integral part in the characteristic exponent of $X^{(\a,\b)}_t$.  
Denoting once again by $(\bar{\a},\bar{\b})$ the unique solution to \eqref{eq:ste15}, we obtain the following (see Pascucci \cite[Page 465]{pascucci2011pde})
\eqstar{
\int_\Rb(\ee^{i \xi y } - 1) \nu (\dd y) = i \lambda \xi  \left( \frac{p}{\lambda_1 - i \eta} - \frac{1-p}{\lambda_2 + i \eta}\right),
}
which in turn gives us
\eqstar{
n_t(\eta) = n^{\bar{\a}}_t(\eta) = \frac{p \lambda}{a} \log\left( \frac{i \eta - \lambda_1}{i \eta \ee^{a t}- \lambda_1}\right) + \frac{(1-p) \lambda}{a} \log\left( \frac{i \eta + \lambda_2}{i \eta \ee^{a t} + \lambda_2}\right).
}
Using this, we can write that
\eqstar{
\bX_t &=\Eb[\cos(X^{(\bar{\a},\bar{\b})}_t)] = \frac{\hat{\mu}_{X^{(\bar{\a},\bar{\b})}_t}(1) + \hat{\mu}_{X^{(\bar{\a},\bar{\b})}_t}(-1)}{2}\\
&=\frac{\ee^{-\frac{1}{2} C_t }}{1 + \ee^{2a t}} \left(\frac{1 + \lambda^2_1}{\ee^{2at} + \lambda^2_1} \right)^{\frac{p \lambda}{2a}} \left(\frac{1 + \lambda^2_2}{\ee^{2at} + \lambda^2_2} \right)^{\frac{(1-p) \lambda}{2a}}\cos\left( m_t + \frac{p \lambda}{a} \bigl( \theta_1 - \theta^\prime_1 \bigr) + \frac{(1-p) \lambda}{a} \bigl( \theta_2 - \theta^\prime_2 \bigr) \right),
}
with $m_t$ and $C_t$ as in \eqref{eq:m_C_True} and where
\eqstar{
\theta_1 &= \arctan\Bigl(-\frac{1}{\lambda_1} \Bigr), &\theta_2 &= \arctan\Bigl(\frac{1}{\lambda_2} \Bigr), &\theta^\prime_1 &= \arctan\Bigl(-\frac{\ee^{at}}{\lambda_1} \Bigr),  &\theta^\prime_2 &= \arctan\Bigl(\frac{\ee^{at}}{\lambda_2} \Bigr).
}
Proceeding now like in Example \ref{example: gaussian}, we obtain a suitable modification of ODE \eqref{eq:ode m simple} for $m_t$, which can be solved numerically to obtain a reference benchmark for $m_t$ and $\bar{\b}_t$.

In Figure \ref{fig:error 1d jump}, like in Example \ref{example: gaussian}, we plot again the quantity in \eqref{eq:errror_mumerics}, and  observe that for parameter values  $a=0.25$, $\sigma = 1.0$, $T = 1.0$, $\lambda = 0.8$, $\lambda_1 = 0.5$, $\lambda_2 = 0.6$, $p = 0.35$, by varying $n = 2^k, 4 \leq k \leq 8$, we confirm the result in Theorem \ref{th:main}. In Figure \ref{fig:comparison 1d jump},  results of the numerical solution from the ODE are compared with the discretized Picard iteration scheme approximation $m^{\alpha^{m,n},\beta^{m,n}}_t$ of $m^{\bar{\a},\bar{\b}}_t$ with $n=2^4$ which once again illustrates the accuracy of our method even for small values of $n$ and $m.$

\begin{figure}[htbp]
\centering
\subfigure[]{\label{fig:error 1d jump}\includegraphics[scale=0.45]{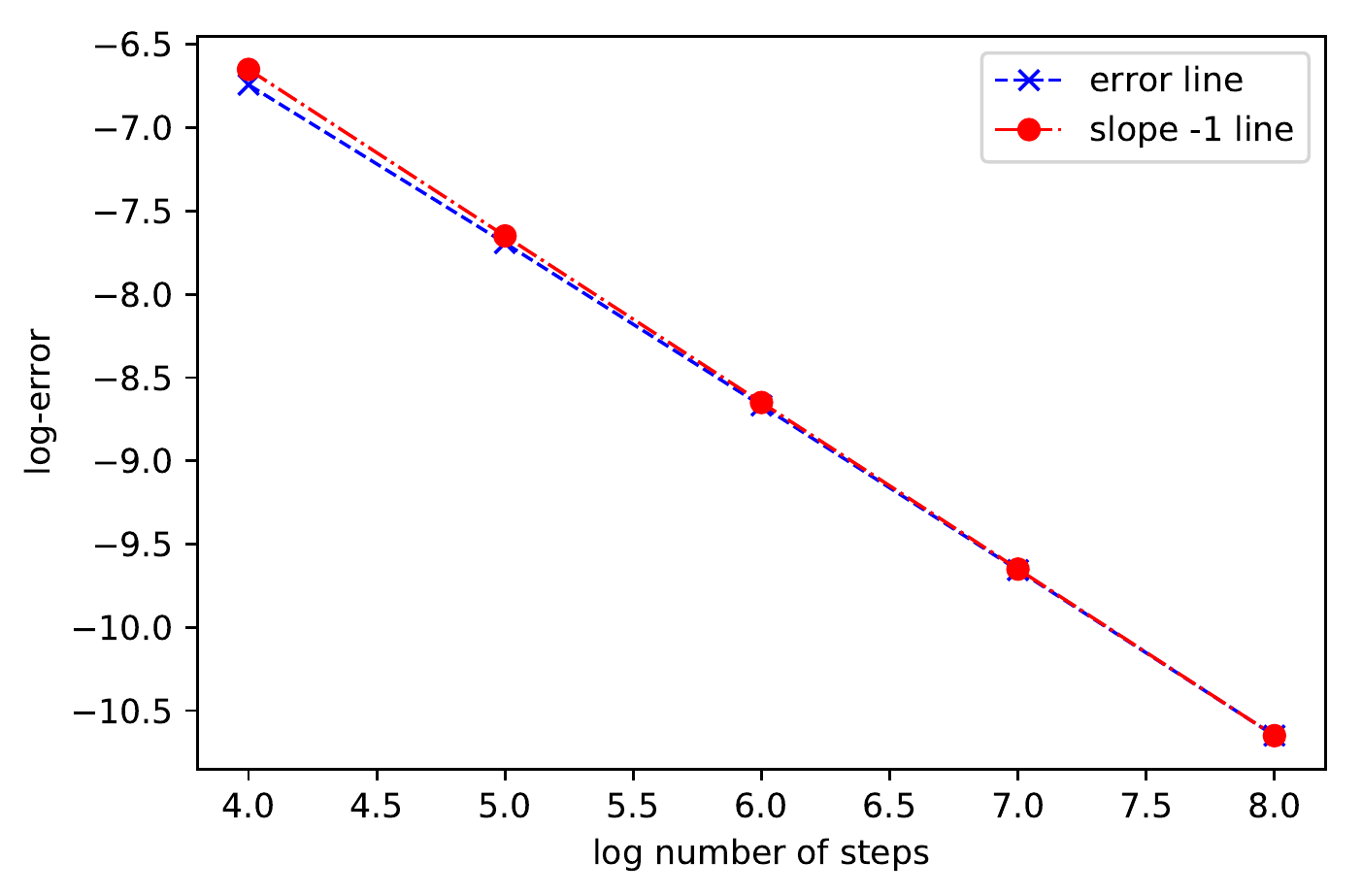}}
~
\subfigure[]{\label{fig:comparison 1d jump}\includegraphics[scale=0.45]{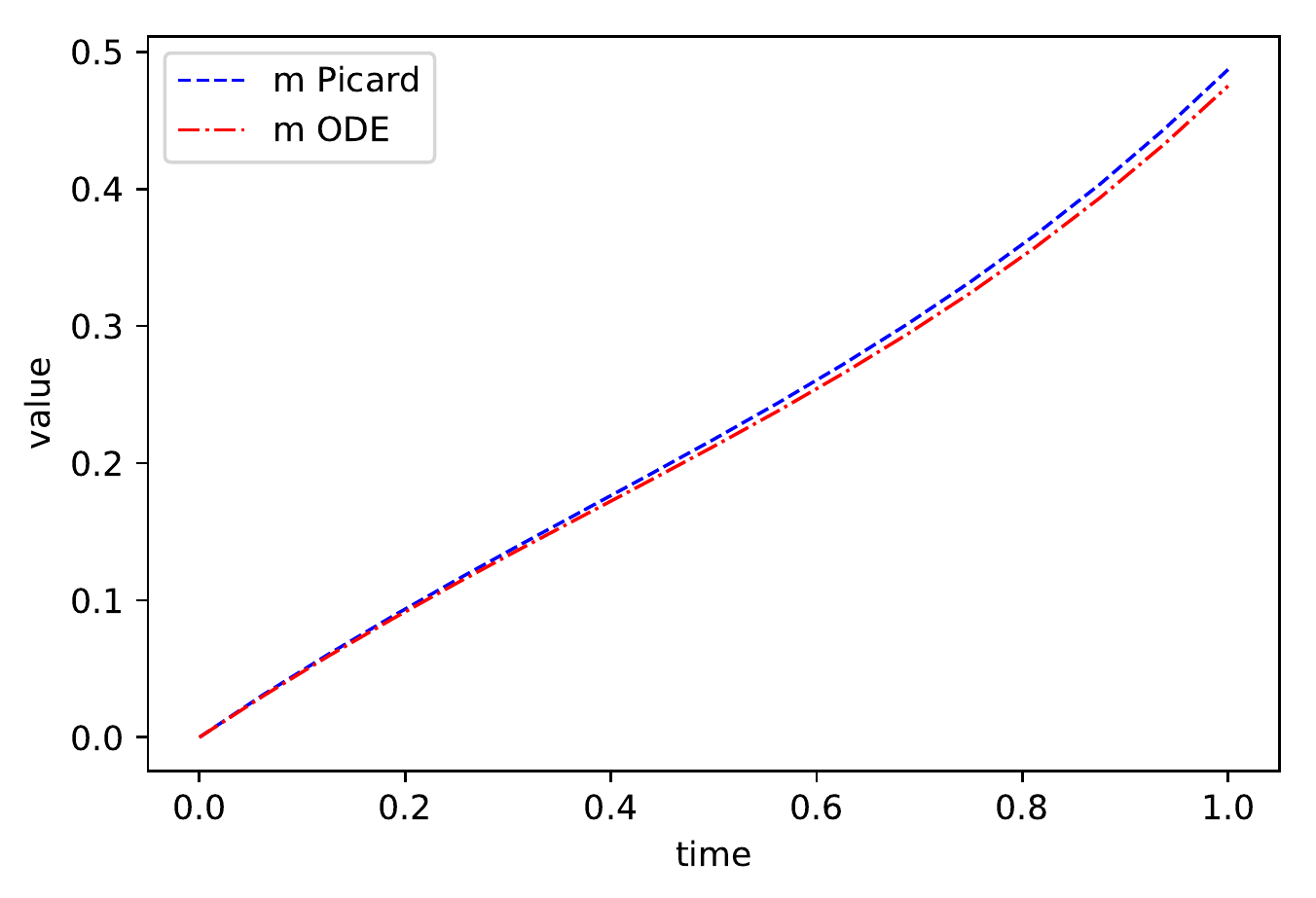}}\\
~
\caption{Error convergence rate and comparison of approximations for the model with jumps.}
\label{fig:1d jump}
\end{figure}

\subsection{Convergence rate in multiple dimensions}
\label{example: multiple dimensions 1}
In this example, we demonstrate that the error convergence rate of the Picard iteration scheme remains independent of the dimension modulo the error of the numerical approximation of the space integral needed to compute $E[b(X^{(\a,\b)}_t)]$. We generalize the MK-V SDE \eqref{eq:example_1D} in a way that the computation of $E[b(X^{(\a,\b)}_t)]$ still does not require numerical integration, and a semi-explicit benchmark (up to solving an ODE) for the solution is still available. We suppose that $Y$ is an $\Rb^d$-valued random variable with a Laplace type distribution given as $\mu_Y(\dd y) = \frac{1}{2^d}\ee^{-\sum^d_{i=1}|y_i|} \dd y$ and consider the following MK-V SDE 
\begin{equation}\label{eq:multi_dim_SDE}
\dd X_t  = \Bigl(
a X_t +  \Eb\big[\cos\big(\textstyle\sum^d_{i=1}X^i_t\big)\big]\mathbf{1} \Bigr) \dd t + \sigma d W_t  , \ t>0, \qquad X_0 = Y,
\end{equation}
where $a\in \R$, $\mathbf{1}=(1,\cdots,1)\in\R^d$, $W$ is a $q$-dimensional Brownian motion, and $\sigma \in \mathcal{M}^{d\times q}$ such that $\theta = \s \s^\top$ is positive definite. Note that \eqref{eq:multi_dim_SDE} can be put in the form \eqref{eq:mckean_SDE_jumps}-\eqref{eq:levy_proc} by setting 
\eqstar{
a(x) \equiv 
a I_d, \quad b(x) =\cos(\textstyle\sum^d_{i=1} x_i)\mathbf{1}, \quad \nu(\dd y) \equiv 0.
}
\begin{figure}[htbp]
\centering
\subfigure[]{\label{fig:error 2d}\includegraphics[scale=0.45]{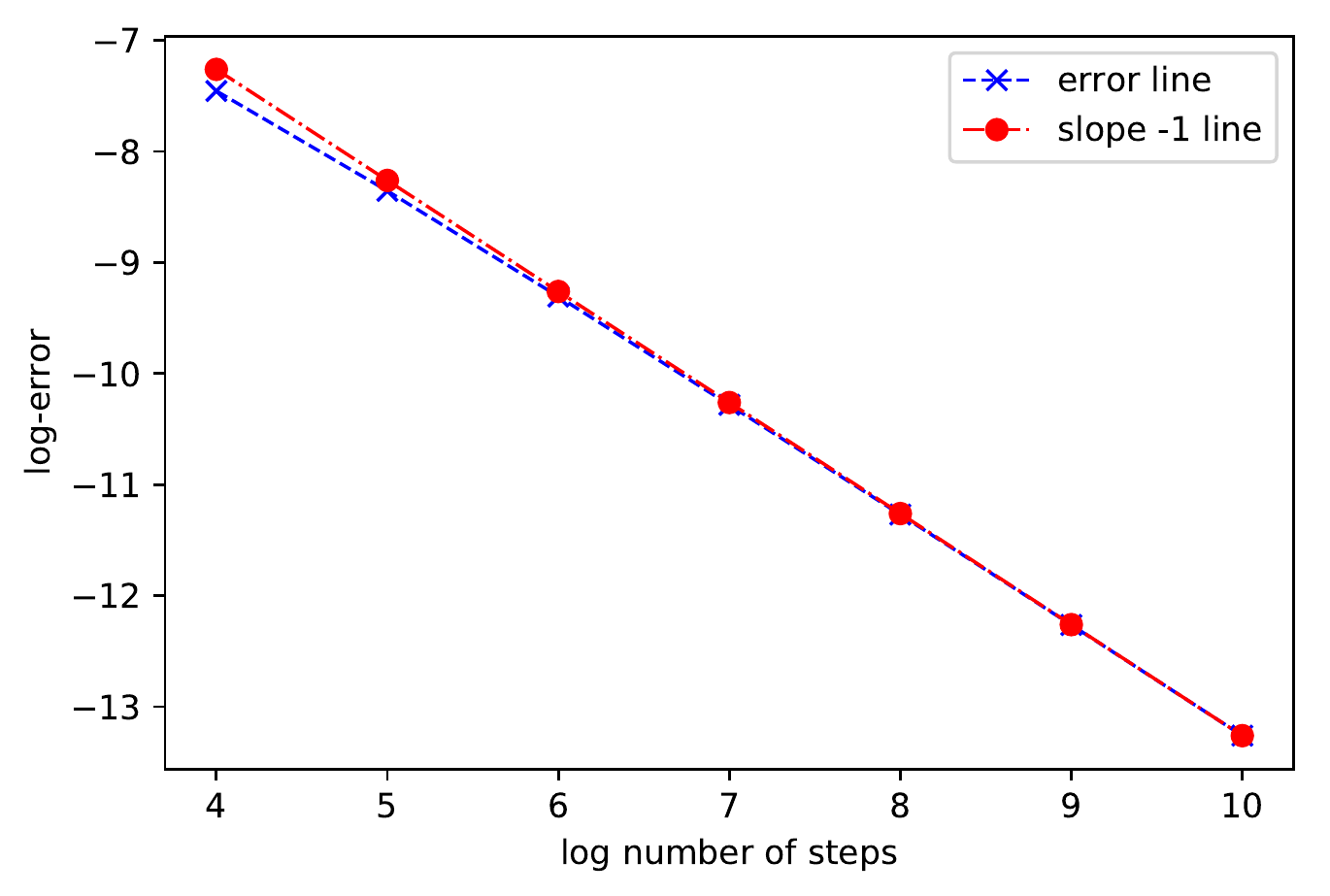}}
~
\subfigure[]{\label{fig:comparison 2d}\includegraphics[scale=0.45]{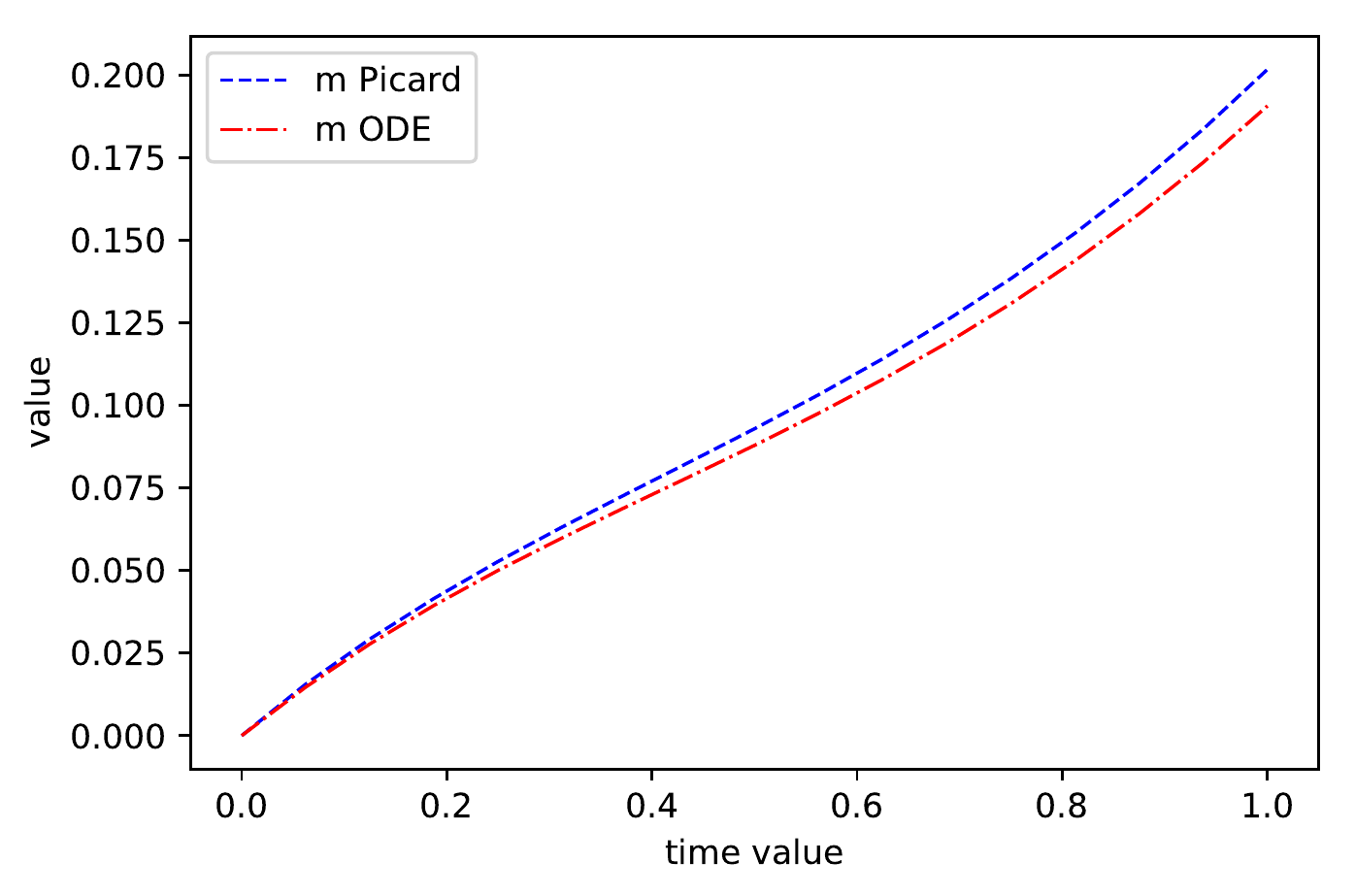}}\\
~
\subfigure[]{\label{fig:error 5d}\includegraphics[scale=0.45]{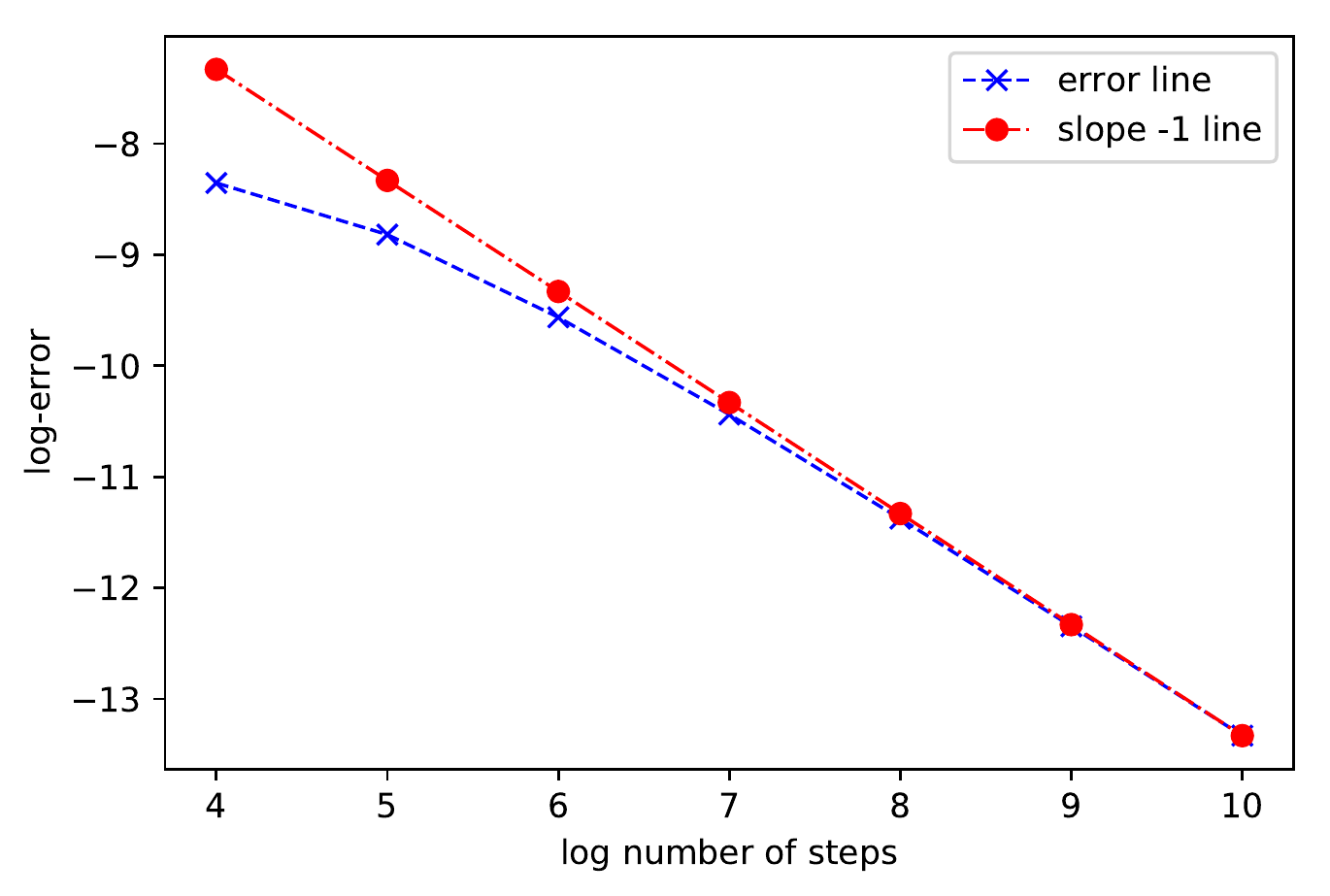}}
~
\subfigure[]{\label{fig:comparison 5d}\includegraphics[scale=0.45]{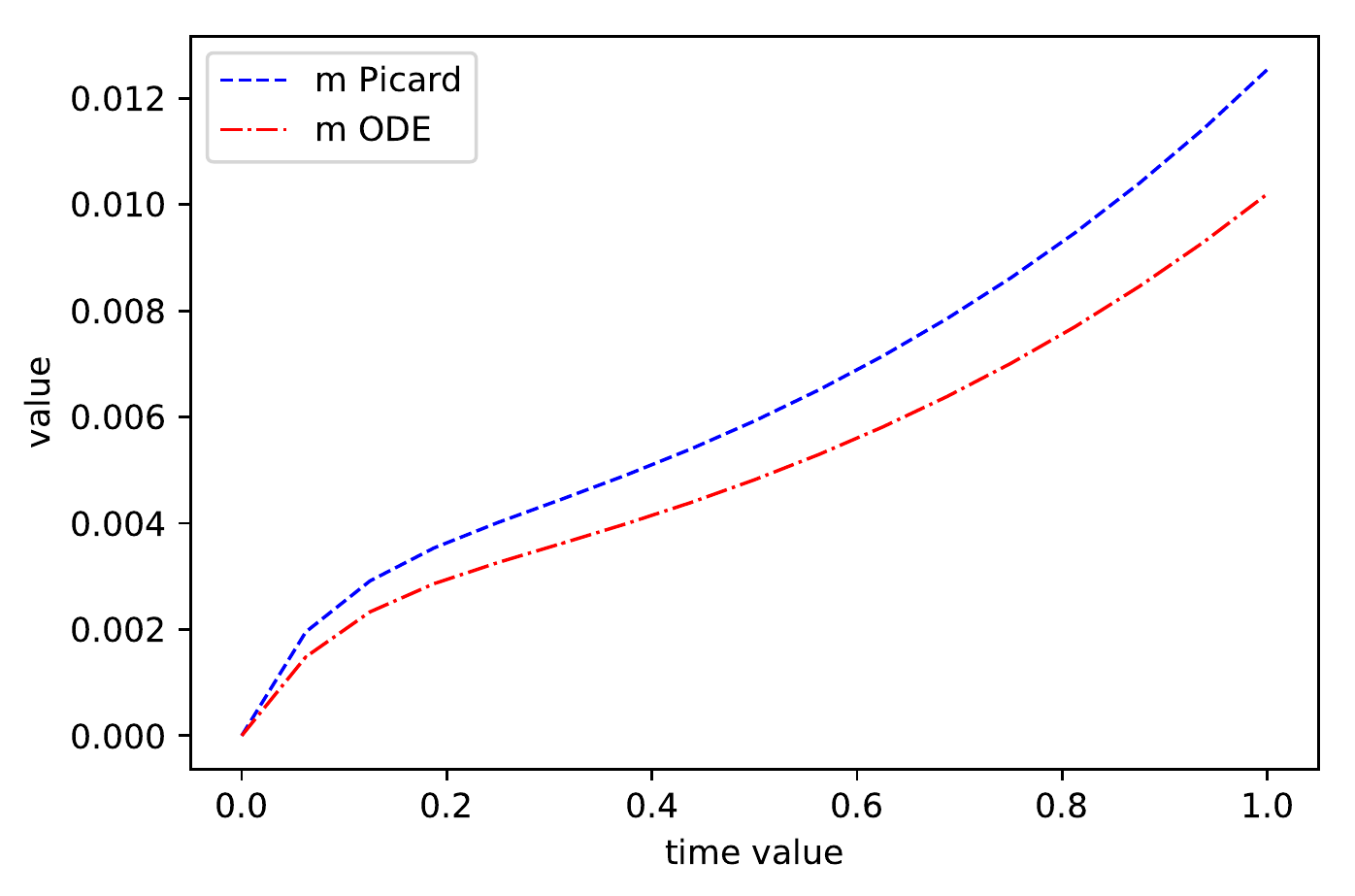}}\\
~
\caption{Error convergence rate (a) $d=2$ (c) $d=5$ and comparison of approximations (b) $d=2$ (d) $d=5$.}
\label{fig:comparison dim}
\end{figure}
The solution $(\bar{\a},\bar{\b})$ to \eqref{eq:ste15} then satisfies the following
\begin{align}
 \bar{\a}_t  & \equiv a I_d  ,  \quad
 \bar{\b}   =   \tilde{\beta}_t \mathbf{1},
\end{align}
where 
\begin{align}
\tilde{\beta}_t &:=  \Eb\big[\cos\big({\textstyle\sum^d_{i=1}}X^{(\bar{\a},\bar{\b}),i}_t\big)\big] =\frac{\hat{\mu}_{X^{(\bar{\a},\bar{\b})}_t}(\mathbf{1}) + \hat{\mu}_{X^{(\bar{\a},\bar{\b})}_t}(-\mathbf{1})}{2}
\intertext{(by Lemma \ref{lem:ito})}
& \,=  \frac{\ee^{-\frac{1}{2}\mathbf{1}^\top{C}_t \mathbf{1}}}{(1+\ee^{2a_0t})^d}\cos(d \, \tilde{m}_t),
\end{align}
with 
\begin{equation}
{C}_t= C^{\bar{\a}}_t =\frac{1}{2a}(\ee^{2a t} - 1 ) \theta, \qquad \tilde{m}_t =  \ee^{a t} \int^t_0\ee^{-a s} \tilde{\beta}_s \dd s.
\end{equation}
In the above, ${C}_t$ is the covariance matrix and $\tilde{m}_t$ is the mean of all the components of $X^{(\bar{\a},\bar{\b})}_t.$ Proceeding once again like in Example \ref{example: gaussian}, we obtain a suitable modification of the ODE \eqref{eq:ode m simple} for $\tilde{m}_t$, which can be solved numerically to obtain a reference benchmark for $\tilde{m}_t$ and $\tilde{\b}_t$. In Figure \ref{fig:comparison dim}, we plot quantities that are analogous to those in Figures \ref{fig:1d} and \ref{fig:1d jump}. We demonstrate that, for $a=0.25$ and randomly generated $\sigma$ matrix, the convergence rate is independent of dimension as proven in Theorem \ref{th:main}, and also provide approximations for $\tilde{m}_t $ using the discretized Picard iteration scheme for $n=2^4.$

\subsection{Convergence rate for non-integrable initial datum}
\label{example: multiple dimensions 2}
In the final example, we consider the case of a non-integrable initial law given by a multivariate $\gamma$-stable distribution with independent components. The characteristic function of $Y$ with $\gamma =1$ is then given as
\eqstar{
 \hat{\mu}_Y(\eta) =\Eb\bigl[\exp(i \langle \eta, Y\rangle) \bigr] =\exp\bigl(i \langle \eta, \mathbf{1} \rangle - \textstyle \sum^d_{k=1} |\eta_k|\bigr).
}
In the above formula, the shift parameter is represented by a unit vector $\mathbf{1}\in \Rb^d,$ and for each component, the skewness parameter is set to zero and the scale parameter is set to one. The above distribution does not admit a first moment. With this example, we wish to test the applicability of our method by relaxing the moment conditions on the initial datum. As expected, it turns out that the finite-moment assumptions on the initial distributions are not essential for the Picard iteration method that we propose (at least in some cases), but they are rather related to the fact that we proved the contraction properties in the Fourier space. 

We work with the  MK-V SDE setting of Section \ref{example: multiple dimensions 1}.
Applying Lemma \ref{lem:ito} it is easy to show that\eqstar{
\tilde{\beta}_t:= \Eb\big[\cos\big({\textstyle\sum^d_{i=1}}X^{(\bar{\a},\bar{\b}),i}_t\big)\big] = \frac{\hat{\mu}_{X^{(\bar{\a},\bar{\b})}_t}(\mathbf{1}) + \hat{\mu}_{X^{(\bar{\a},\bar{\b})}_t}(-\mathbf{1})}{2} = \ee^{-\frac{1}{2}\mathbf{1}^\top{C}_t \mathbf{1} - d \ee^{a t}} \cos \bigl( d(\ee^{a t} + \tilde{m}_t)\bigr),
}
and a benchmark for $\tilde{\beta}_t $ can be obtained by solving a suitable modification of the ODE \eqref{eq:ode m simple}. In Figure \ref{fig:comparison datum}, we demonstrate that even though the initial datum has an undefined mean, our Picard iteration scheme converges and the rate is independent of dimension as proven in Theorem \ref{th:main}. We also provide approximation for $\tilde{m}_t $ using the discretized Picard iteration scheme for $n=2^4.$

\begin{figure}[htbp]
\centering
\subfigure[]{\label{fig:error 2d stable}\includegraphics[scale=0.45]{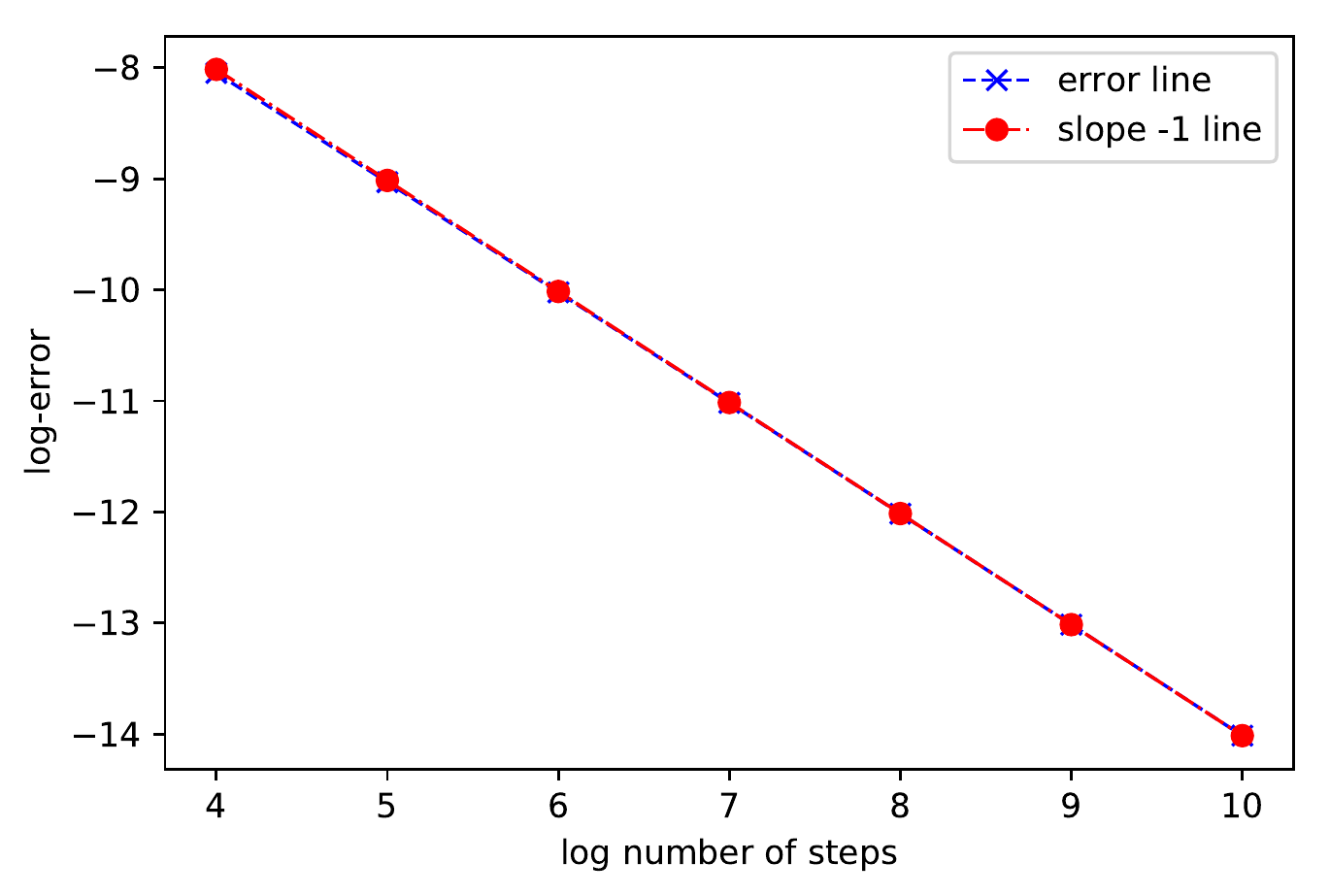}}
~
\subfigure[]{\label{fig:comparison 2d stable}\includegraphics[scale=0.45]{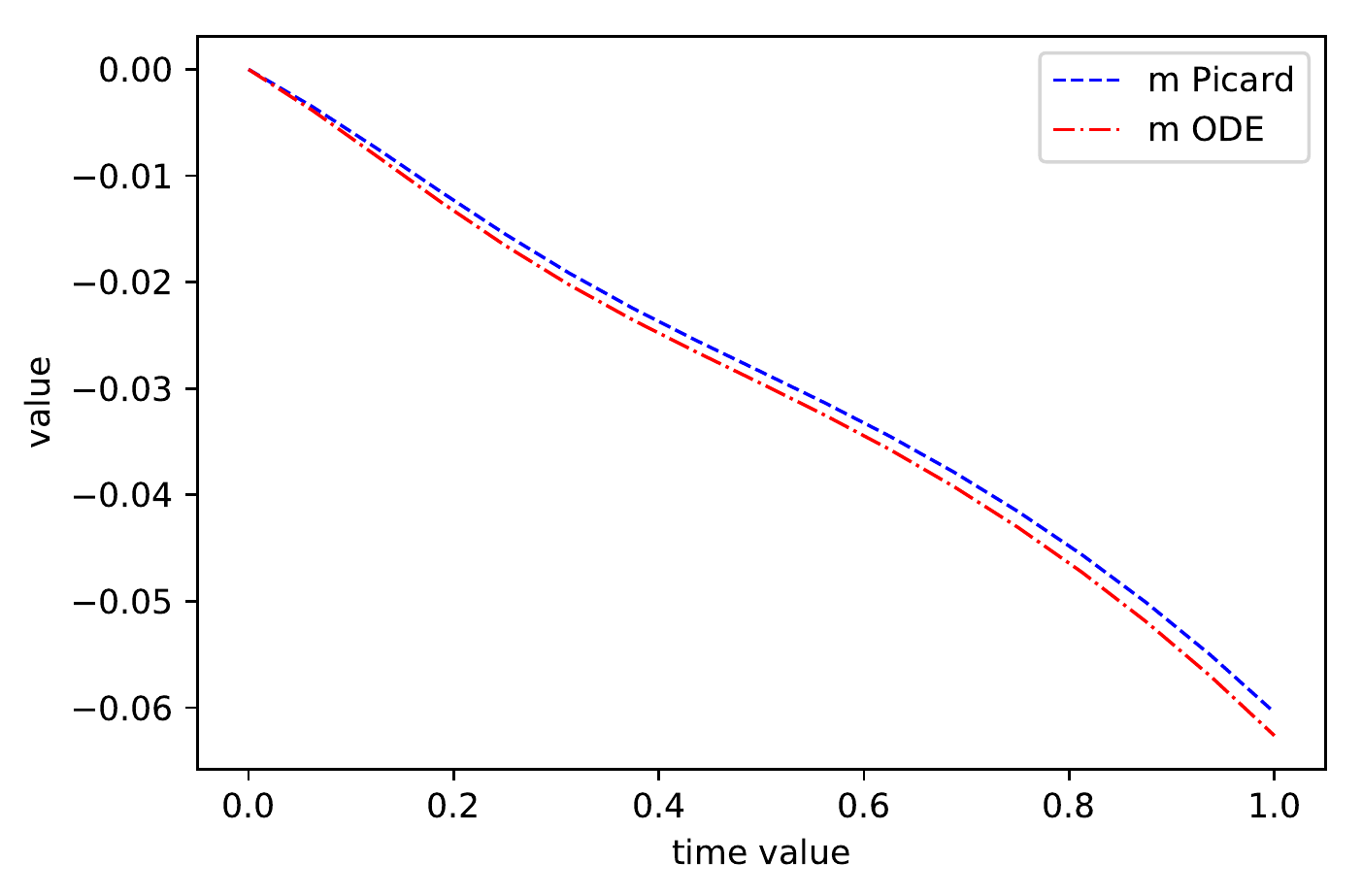}}\\
~
\caption{Error convergence rate and comparison of approximations for the model with initial datum with an undefined mean in $d=2.$}
\label{fig:comparison datum}
\end{figure}

\section*{Acknowledgement}
The authors would like to thank Prof.\ Emmanuel Gobet for the useful initial discussions on the theoretical results and convergence of the numerical scheme.

\appendix

\section{Proofs of Lemma \ref{lem:ito}, \ref{lemm:estimates_bis} and \ref{lemm:estimates_ter}
}
\label{app:proof_ito}
\begin{proof}[Proof of Lemma \ref{lem:ito}.]
To shorten notation, throughout the proof, we set ${X}_t = {X}^{(\a,\b)}_t$ and $\Phi_{s,t} = \Phi^{\alpha}_{s,t}$. 

Let us set $\tilde{X}_t=\Phi^{-1}_{0,t}X_t$, with $\Phi$ as in \eqref{eq:ode}. By It\^{o}'s formula we obtain that
\begin{equation}\label{eq:ste26}
\tilde{X}_t = Y + \int_0^t \Phi^{-1}_{0,s} \b_s \dd s + \int_0^t 
\Phi^{-1}_{0,s} \dd L_s,
\end{equation}
and thus, we have that
\begin{equation}\label{eq:ste24}
\hat{\mu}_{X_t}( \eta)=\Eb\big[ \exp\big(  i \big\langle \eta,
\Phi_{0,t}\tilde{X}_t  \big\rangle  \big) \big]=\Eb\big[ \exp\big(  i \big\langle \Phi^\top_{0,t} \eta ,\tilde{X}_t  \big\rangle  \big) \big]
= \Eb\big[ \exp\big(  i \big\langle \Phi^\top_{0,t} \eta, 
\tilde{X}_t - Y\big\rangle  \big) \big] \hat{\mu}_Y \big( \Phi^\top_{0,t}\eta\big).
\end{equation}
Set $Z_t:= \ee^{i \langle \xi, \tilde{X}_t - Y\rangle}.$ By  It\^{o}'s formula, we obtain that
\begin{align}
 Z_t &=1+ \int_0^t Z_s  \int_{\R^d} \Big( \ee^{{i \langle \xi , \Phi^{-1}_{0,s} y\rangle}} -1 -  {\bf 1}_{\{|y|<1\}} i\big\langle  \xi,  \Phi^{-1}_{0,s} y\big\rangle  \Big) \nu(\dd y)  \dd s  \\
 &\quad +  \int_0^t Z_s \Big(   i\big\langle  \xi,  \Phi^{-1}_{0,s} \b_s\big\rangle -\frac{1}{2} \big\langle    \xi ,\Phi^{-1}_{0,s} \s \big( \Phi^{-1}_{0,s} \s\big)^\top \xi  \big\rangle    \Big) \dd s    \\
 &\quad
 + i \int_0^t Z_s  \xi^\top  \Phi^{-1}_{0,s} \s \dd W_s  + \int_0^t Z_{s-} \int_{\R^d}  \Big( \ee^{{i \langle \xi , \Phi^{-1}_{0,s} y\rangle}} -1  \Big)\big({N}(\dd s, \dd y) - \nu (\dd y)\dd s\big).
\end{align}
Now set $\phi_t(\xi):=\Eb[ Z_t ]$. By the martingale property of the  It\^{o} and the jump integrals, combined with Fubini's theorem, we obtain that
$$
\phi_t(\xi) = 1 + \int_0^t \phi_s(\xi) \psi_s(\xi) \dd s, 
$$
where
$$
\psi_s(\xi) = \int_{\R^d} \Big( \ee^{{i \langle \xi , \Phi^{-1}_{0,s} y\rangle}} -1 -  {\bf 1}_{\{|y|<1\}} i\big\langle  \xi,  \Phi^{-1}_{0,s} y\big\rangle  \Big) \nu(\dd y) +  i\big\langle  \xi,  \Phi^{-1}_{0,s} \b_s\big\rangle -\frac{1}{2} \big\langle    \xi ,\Phi^{-1}_{0,s} \s \big( \Phi^{-1}_{0,s} \s\big)^\top \xi  \big\rangle.
$$
By differentiating both terms, we have that 
\begin{align}
\frac{\dd}{\dd t} \phi_t(\xi) = \phi_t(\xi) \psi_t(\xi), \quad t>0, \qquad \phi_0(\xi) = 1,
\end{align}
which yields that
$$
 \Eb\big[ \exp\big(  i \big\langle \xi, \tilde{X}_t - Y\big\rangle  \big) \big] = \phi_t(\xi) =  \ee^{\int_0^t \psi_s(\xi)\dd s},
$$
which in turn, combined with \eqref{eq:ste24}, yields \eqref{eq:ste13} and concludes the proof.
\end{proof}

\begin{proof}[Proof of Lemma \ref{lemm:estimates_bis}]
To ease the notation, we set $\Phi_{s,t}=\Phi^{\a}_{s,t}$, $C_t=C^{\a}_t$, $m_t=m^{\a,\beta}_t$ and $n_t(\eta)=n^{\a}_t(\eta)$. 

The inequality \eqref{eq:est_1} is a straightforward consequence of Gr\"onwall's Lemma. We now prove \eqref{eq:est_1_b}. By \eqref{eq:ode} it holds that 
\begin{equation}\label{eq:cauchy_C}
\left\{
\begin{aligned}
&\frac{\dd}{\dd t}C_{t}    =\theta + 2 \alpha_t C_{t} , \quad s<t\leq T, \\
&\Phi_{0}=0,
\end{aligned}
\right.
\end{equation}
which means that $C$ is absolutely continuous and $C_t =t \theta +2 \int_{0}^t \a_u C_u \dd u$. Using Gr\"onwall's Lemma again yields \eqref{eq:est_1_b}. The proof of \eqref{eq:est_1_c} is completely analogous. 

To prove \eqref{eq:estim_n}, set 
\begin{equation}\label{eq:def_f}
f(x):= \int_{\R^d} \Big( \ee^{i \langle x , y\rangle} -1 -  {\bf 1}_{\{|y|<1\}} i\langle   x, y\rangle  \Big) \nu(\dd y), \quad x\in\R^d.
\end{equation}
Now, for any $0\leq \delta<1$, Taylor's theorem yields that
\begin{equation}
|f(x) | \leq 2 \bigg(  |x|^2  \int_{|y|< \delta} |y|^2 \nu(\dd y)  + ( |x| +1 ) \int_{|y|\geq \delta} \nu(\dd y) \bigg), \quad x\in\R^d,
\end{equation}
which, combined with \eqref{eq:est_1} gives that
\begin{equation}\label{eq:upper_f}
\big|  f \big(  \Phi_{s,t}^\top \eta  \big)  \big| \leq 2  \ee^{ 2T \|\alpha\|_{T,0} }\bigg(  |\eta|^2  \int_{|y|< \delta} |y|^2 \nu(\dd y)  + (|\eta|+1) \int_{|y|\geq \delta}  \nu(\dd y) \bigg).
\end{equation}
This, together with the fact that
\begin{equation}\label{eq:rep_f}
n_t(\eta) = \int_0^t  f \big( \Phi_{s,t}^\top \eta  \big)  \dd s,
\end{equation}
proves \eqref{eq:estim_n}.

We finally prove \eqref{eq:estimate_quad_form}. We have that
\begin{equation}\label{eq:estim_proof_1}
|\langle \eta, C_t\eta\rangle| \geq \lambda_{\text{min}}(\theta)  \int_{0}^t    \eta^\top  \Phi_{s,t} \Phi_{s,t}^\top   \eta \,  \dd s,
\end{equation}
and
\begin{equation}\label{eq:estim_proof_2}
\lambda_{\text{min}}\Big(  \Phi_{s,t} \Phi_{s,t}^\top  \Big) =\frac{1}{ \lambda_{\text{max}}\Big( \big( \Phi_{s,t}  \Phi_{s,t}^\top \big)^{-1}  \Big)} = \frac{1}{  \big| \big( \Phi_{s,t} \big)^{-1} \big|^2  }.
\end{equation}
Noting that $\Phi_{s,t}^{-1} = I_d - \int_0^t   \Phi_{s,u}^{-1} \alpha_u \dd u $, and applying Gr\"onwall's Lemma yields that
\begin{equation}\label{eq:estim_proof_3}
\big| \Phi_{s,t}^{-1} \big| \leq \ee^{T \|\a \|_{T,0}}.
\end{equation}
Eventually, \eqref{eq:estimate_quad_form} results from \eqref{eq:estim_proof_1}--\eqref{eq:estim_proof_3}.
\end{proof}

\begin{proof}[Proof of Lemma \ref{lemm:estimates_ter}.]

To ease the notation, we set $\Phi_{s,t}=\Phi^{\a}_{s,t}$, $C_t=C^{\a}_t$, $m_t=m^{\a,\beta}_t$ and $n_t(\eta)=n^{\a}_t(\eta)$. 

By definition of $\Phi,$ we obtain that
\begin{equation}
|\Phi_{s,t} - \Phi_{s,t'}| = \Big|   \int_{t}^{t'} \Phi_{s,u} \a_u  \dd u \Big|  \leq \int_{t}^{t'}  |\Phi_{s,u}|  |\a_u|   \dd u ,
\end{equation}
and \eqref{eq:est_1_ter} follows from \eqref{eq:est_1}.

We now prove \eqref{eq:est_1_b_ter}. By definition of $C$ and by triangular inequality, we obtain that
\begin{equation}\label{eq:ctminuscprime}
|C_t - C_{t'}| \leq  \int_t^{t'} \big| \Phi_{s,t'} \s \s^\top \Phi_{s,t'}^\top  \big|  \dd s + \int_0^t \big| \big(\Phi_{s,t'} \s \s^\top \Phi_{s,t'}^\top   \big)  -  \big(\Phi_{s,t} \s \s^\top \Phi_{s,t}^\top   \big)  \big| \dd s.
\end{equation}
Applying \eqref{eq:est_1} and \eqref{eq:est_1_ter} yields that
\begin{equation}
\big| \Phi_{s,t'} \s \s^\top \Phi_{s,t'}^\top  \big|  \leq   \ee^{2 T \| \a \|_{T,0}} |\theta|,
\end{equation}
and
\begin{align}
\big| \big(\Phi_{s,t'} \s \s^\top \Phi_{s,t'}^\top   \big)  -  \big(\Phi_{s,t} \s \s^\top \Phi_{s,t}^\top   \big)  \big| &\leq 
\big| \Phi_{s,t'} \s \s^\top (\Phi_{s,t'}-\Phi_{s,t})^\top  \big| + 
\big| (\Phi_{s,t'}-\Phi_{s,t} ) \s \s^\top \Phi_{s,t}^\top  \big| \\
& \leq 2 \| \a \|_{T,0}  \ee^{2T\| \a \|_{T,0}}  |\theta| (t'-t).
\end{align}
Plugging these into \eqref{eq:ctminuscprime}, we have that
\begin{equation}
|C_t - C_{t'}| \leq (1+2 T \| \a \|_{T,0}) \ee^{2 T \| \a \|_{T,0}} |\theta| (t'-t), 
\end{equation}
which proves \eqref{eq:est_1_b_ter}. The proof of \eqref{eq:est_1_c_ter} is analogous and thus, is omitted. 

We finally prove \eqref{eq:estim_n_ter}. Consider the function $f$ as defined in \eqref{eq:def_f}. For any $M>0$, an application of Lagrange's mean-value theorem yields the following
\begin{equation}\label{eq:ste150}
|f(x)-f(x')| \leq  
2 (M+1)\, \bar{n}\,|x-x'|, \quad |x|,|x'|<M,
\end{equation}
with $\bar{n}$ as in \eqref{eq:ste52}. This, together with \eqref{eq:est_1}, yields that
\begin{align}
\big| f\big(  \Phi_{s,t'}^\top \eta  \big)-  f\big( \Phi_{s,t}^\top \eta  \big)  \big| 
&\leq 2 \big( |\eta| \ee^{ T \| \a \|_{T,\lambda}} +1 \big) \bar{n}   \big|  ( \Phi_{s,t'}^\top - \Phi_{s,t}^\top ) \eta \big| 
\intertext{(by \eqref{eq:est_1_ter})}
&\leq (t'-t)2\|\a \|_{T,0} \ee^{2T \|\a \|_{T,0}} \big(|\eta|^2+|\eta|\big)  \bar{n}, \label{eq:estimate_f_b}
\end{align}
while \eqref{eq:upper_f} gives the following
\begin{equation} \label{eq:estimate_f_c}
\big| f\big(  \Phi_{s,t'}^\top \eta  \big)  \big| \leq 2  \ee^{ 2T \|\alpha\|_{T,0} } \big(|\eta|^2+|\eta|+1\big) \bar{n}.
\end{equation}
Now, by \eqref{eq:rep_f} and by triangular inequality we obtain that
\begin{align}
  | n_t(\eta) - n_{t'}(\eta) | &\leq \int_{t}^{t'} \big| f\big( \Phi_{s,t'}^\top \eta  \big) \big| \dd s + \int_0^t \big| f\big( \Phi_{s,t'}^\top \eta  \big)-  f\big( \Phi_{s,t}^\top \eta  \big) \big| \dd s
  \intertext{(by \eqref{eq:estimate_f_b}-\eqref{eq:estimate_f_c})}
 &\leq 2(1+T  \|\alpha\|_{T,0}) \ee^{ 2T \|\alpha\|_{T,0} } \big(|\eta|^2+|\eta|+1\big) \bar{n} (t'-t),
  \end{align}
which proves \eqref{eq:estim_n_ter} and concludes the proof.

\end{proof}

\section{Proofs of Lemma \ref{lem:estimates} and Lemma \ref{lemm:estim_der_n}}
\label{app:estimates_apriori}

\begin{proof}[Proof of Lemma \ref{lem:estimates}.]
We first show \eqref{eq:ste104}. Set $w:=C^{\alpha}-C^{\alpha'}$. By \eqref{eq:cauchy_C}, we obtain that $w$ is the solution to the following equation
\begin{equation}\label{eq:ode_ter}
\left\{
\begin{aligned}
&w'_t    = 2 \big( \alpha_t C^{\alpha}_t - \alpha'_t C^{\alpha'}_t \big), \quad t>0\\
&w_0=0,
\end{aligned}
\right.
\end{equation}
which means that $w$ is absolutely continuous and is given as 
\begin{equation}
w_t = 2 \int_0^t  \big( \alpha_s C^{\alpha}_s - \alpha'_s C^{\alpha'}_s \big)  \dd s = 2 \int_0^t  \Big( \alpha_s w_s + \big(\alpha_s - \alpha'_s \big) C^{\alpha'}_s \Big)  \dd s,
\end{equation}
and thus, satisfies the following
\begin{equation}
|w_t| \leq 2 \int_0^t \Big( | \alpha_s | | w_s | + | \alpha_s - \alpha'_s | \big| C^{\alpha'}_s \big| \Big)  \dd s.
\end{equation}
Now, by Gr\"onwall's Lemma we obtain that
\begin{equation}
|w_t| \leq 2 \,\ee^{2 \int_0^t  | \alpha_s |   \dd s} \int_0^t | \alpha_s - \alpha'_s | \big| C^{\alpha'}_s \big|   \dd s ,
\end{equation}
which, together with \eqref{eq:est_1_b}, yields that
\begin{equation}\label{eq:ste105}
|w_t| \leq  2 t |\theta| \,\ee^{2 T ( \|\alpha\|_{T,0}  +  \|\alpha'\|_{T,0} )   } \int_0^t | \alpha_s - \alpha'_s |   \dd s  \leq  2 t |\theta| \,\ee^{4 T  \| a \|_{L^{\infty}}   } \int_0^t | \alpha_s - \alpha'_s |   \dd s.
\end{equation}
Now, combining the following 
\begin{equation}
|\langle \eta, C^{\alpha}_t\eta\rangle - \langle \eta, C^{\alpha'}_t\eta\rangle  | \leq |\eta| \big| \big( C^{\alpha}_t - C^{\alpha'}_t \big)\eta \big| \leq |\eta|^2 \big| C^{\alpha}_t - C^{\alpha'}_t \big| = \,|\eta|^2 |w_t|
\end{equation}
with \eqref{eq:ste105} proves \eqref{eq:ste104}. 

The proof of \eqref{eq:ste104_quin} is analogous, and so are the proofs of 
\eqref{eq:ste104bis} and \eqref{eq:ste104ter} after observing that 
\begin{equation}
\left\{
\begin{aligned}
&\frac{\dd}{\dd t}  \big(m^{\a,\beta}_t - m^{\a',\beta}_t \big)  = \alpha_t m^{\alpha,\beta}_t - \alpha'_t m^{\alpha',\beta}_t , \quad t>0,\\
&m^{\a,\beta}_0 - m^{\a',\beta}_0=0,
\end{aligned}
\right.
\end{equation}
\begin{equation}
\left\{
\begin{aligned}
&\frac{\dd}{\dd t}  \big(m^{\a,\beta}_t - m^{\a,\beta'}_t \big)  = \alpha_t \big(m^{\a,\beta}_t - m^{\a,\beta'}_t \big) + (\beta_t -\beta'_t) , \quad t>0,\\
&m^{\a,\beta}_0 - m^{\a,\beta'}_0=0.
\end{aligned}
\right.
\end{equation}

We next show \eqref{eq:ste104quat}. By \eqref{eq:ste150} and \eqref{eq:est_1} we obtain that
\begin{align}
\big| f\big(  (\Phi^{\a}_{s,t})^\top \eta  \big)-  f\big( (\Phi^{\a'}_{s,t})^\top \eta  \big)  \big| &\leq 2 \big( |\eta| \ee^{ T \| a \|_{L^{\infty}}} +1 \big) \bar{n}   \big|   (\Phi^{\a}_{s,t})^\top \eta - (\Phi^{\a'}_{s,t})^\top \eta \big| 
\intertext{(by \eqref{eq:ste104_quin})}
&\leq \kappa  |\eta| (|\eta| + 1)  \int_0^t | \alpha_s - \alpha'_s | \dd s,
\end{align}
which in turn, combined with the following
$$
n^{\alpha}_t(\eta) - n^{\alpha'}_t(\eta) = \int_0^t \Big( f\big( (\Phi^{\a}_{s,t})^\top \eta  \big)-  f\big( (\Phi^{\a'}_{s,t})^\top \eta  \big) \Big) \dd s,
$$
yields \eqref{eq:ste104quat} and concludes the proof.
\end{proof}

\begin{proof}[Proof of Lemma \ref{lemm:estim_der_n}.] By \eqref{eq:der_n_first} we have that
\begin{align}
\partial_{\eta_j}\big( n^{\alpha}_t- n^{\alpha'}_t\big)(\eta) & = i \int_{0}^t  \int_{\R^d}  \big( \Phi^{\alpha}_{s,t} y -   \Phi^{\alpha'}_{s,t} y \big)_j   \Big( \ee^{i \langle  \eta , \Phi^{\alpha}_{s,t} y\rangle} -{\bf 1}_{\{|y|<1\}}\Big)   \nu(\dd y)   \dd s,   \\
&\quad +  i \int_{0}^t   \int_{\R^d}  \big( \Phi^{\alpha'}_{s,t} y \big)_j \Big( \ee^{i \langle  \eta , \Phi^{\alpha}_{s,t} y\rangle} -\ee^{i \langle  \eta , \Phi^{\alpha'}_{s,t} y\rangle}\Big) \nu(\dd y)  \dd s,
\end{align}
for any $j=1,\cdots,d$, from which we obtain that
\begin{align}
|\partial_{\eta_j}\big( n^{\alpha}_t- n^{\alpha'}_t\big)(\eta)| & \leq  \int_{0}^t  \int_{\R^d}  \big| \Phi^{\alpha}_{s,t}  -   \Phi^{\alpha'}_{s,t}\big| |y|   \Big( \big| \Phi^{\alpha}_{s,t}  \big| |y| |\eta| +{\bf 1}_{\{|y|\geq1\}}\Big)   \nu(\dd y)   \dd s,   \\
&\quad +   \int_{0}^t   \int_{\R^d}  \big| \Phi^{\alpha'}_{s,t}  \big| |y|^2 |\eta|  \big| \Phi^{\alpha}_{s,t}  -   \Phi^{\alpha'}_{s,t}\big| \nu(\dd y)  \dd s 
\intertext{by \eqref{eq:est_1}}
& \leq 2\big( \ee^{T \|\a \|_{T,0}}+\ee^{T \|\a' \|_{T,0}}\big) (|\eta| +1)  \int_{\R^d}|y|^2 \nu(\dd y) \int_{0}^t    \big| \Phi^{\alpha}_{s,t}  -   \Phi^{\alpha'}_{s,t}\big| \dd s.
\end{align}
The estimate \eqref{eq:ste104quat_bis} for $m=1$ then follows from \eqref{eq:ste104quat}. The proof for $m>1$ is completely analogous and thus, it is omitted.
\end{proof}
\bibliographystyle{acm}
\bibliography{mckeanfourier}
\end{document}